\documentclass[11pt]{amsart}

\usepackage{amssymb}
\usepackage{url}

\usepackage{wasysym}
\usepackage{amsmath}
\usepackage{tikz-cd}

\usepackage{mathrsfs,multicol,float}
\usepackage[ hypertexnames=false, colorlinks, citecolor=red, linkcolor=red]{hyperref}

	\normalbaselines						  %
\def\RR{\mathbb{R}}

\def\CC{\mathbb{C}}

\newcommand{\al}{{\alpha}}

\newcommand{\f}{{\varphi}}

\newcommand{\R}{{\mathbb  R}}

\newcommand{\C}{{\mathbb  C}}

\newcommand{\wt}{\widetilde}

\newcommand{\pt}{\partial}

\makeatletter
\def\Ddots{\mathinner{\mkern1mu\raise\p@
\vbox{\kern7\p@\hbox{.}}\mkern2mu
\raise4\p@\hbox{.}\mkern2mu\raise7\p@\hbox{.}\mkern1mu}}
\makeatother

\newcommand{\cC}{\mathcal{C}}

\newcommand{\cF}{\mathcal{F}}

\newcommand{\cS}{\mathcal{S}}

\newcommand{\eps}{\varepsilon}

\newcommand{\cT}{\mathcal{T}}

\DeclareMathOperator{\dom}{dom}
\DeclareMathOperator{\ran}{ran}

\catcode`@=11
\chardef\mathlig@atcode\count255

\def\actively#1#2{\begingroup\uccode`\~=`#2\relax\uppercase{\endgroup#1~}}
\def\mathlig@gobble{\afterassignment\mathlig@next@cmd\let\mathlig@next= }
\def\mathlig@delim{\mathlig@delim}
\def\mathlig@defcs#1{\expandafter\def\csname#1\endcsname}
\def\mathlig@let@cs#1#2{\expandafter\let\expandafter#1\csname#2\endcsname}
\def\mathlig@appendcs#1#2{\expandafter\edef\csname#1\endcsname{\csname#1\endcsname#2}}

\def\mathlig#1#2{\mathlig@checklig#1\mathlig@end\mathlig@defcs{mathlig@back@#1}{#2}\ignorespaces}

\def\mathlig@checklig#1#2\mathlig@end{%
 \expandafter\ifx\csname mathlig@forw@#1\endcsname\relax
 \expandafter\mathchardef\csname mathlig@back@#1\endcsname=\mathcode`#1%
 \mathcode`#1"8000\actively\def#1{\csname mathlig@look@#1\endcsname}%
 \mathlig@dolig#1\mathlig@delim
\fi
\mathlig@checksuffix#1#2\mathlig@end
}

\def\mathlig@checksuffix#1#2\mathlig@end{%
\ifx\mathlig@delim#2\mathlig@delim\relax\else\mathlig@checksuffix@{#1}#2\mathlig@end\fi
}
\def\mathlig@checksuffix@#1#2#3\mathlig@end{%
\expandafter\ifx\csname mathlig@forw@#1#2\endcsname\relax\mathlig@dosuffix{#1}{#2}\fi
\mathlig@checksuffix{#1#2}#3\mathlig@end
}

\def\mathlig@dosuffix#1#2{%
\mathlig@appendcs{mathlig@toks@#1}{#2}%
\mathlig@dolig{#1}{#2}\mathlig@delim
}

\def\mathlig@dolig#1#2\mathlig@delim{%
 \mathlig@defcs{mathlig@look@#1#2}{%
 \mathlig@let@cs\mathlig@next{mathlig@forw@#1#2}\futurelet\mathlig@next@tok\mathlig@next}%
 \mathlig@defcs{mathlig@forw@#1#2}{%
  \mathlig@let@cs\mathlig@next{mathlig@back@#1#2}%
  \mathlig@let@cs\checker{mathlig@chck@#1#2}%
  \mathlig@let@cs\mathligtoks{mathlig@toks@#1#2}%
  \expandafter\ifx\expandafter\mathlig@delim\mathligtoks\mathlig@delim\relax\else
  \expandafter\checker\mathligtoks\mathlig@delim\fi
  \mathlig@next
 }%
 \mathlig@defcs{mathlig@toks@#1#2}{}%
 \mathlig@defcs{mathlig@chck@#1#2}##1##2\mathlig@delim{%
  \ifx\mathlig@next@tok##1%
   \mathlig@let@cs\mathlig@next@cmd{mathlig@look@#1#2##1}\let\mathlig@next\mathlig@gobble
  \fi
  \ifx\mathlig@delim##2\mathlig@delim\relax\else
   \csname mathlig@chck@#1#2\endcsname##2\mathlig@delim
  \fi
 }%
 \ifx\mathlig@delim#2\mathlig@delim\else
  \mathlig@defcs{mathlig@back@#1#2}{\csname mathlig@back@#1\endcsname #2}%
 \fi
}%

\catcode`@\mathlig@atcode

\mathchardef\ordinarycolon\mathcode`\:
\def\vcentcolon{\mathrel{\mathop\ordinarycolon}}
\mathlig{:=}{\vcentcolon=}
\mathlig{::=}{\vcentcolon\vcentcolon=}

\numberwithin{equation}{section}

\theoremstyle{plain}
\newtheorem{theo}{Theorem}[section]
\newtheorem{cor}[theo]{Corollary}
\newtheorem{lem}[theo]{Lemma}
\newtheorem{prop}[theo]{Proposition}

\newtheorem{hypothesis}[theo]{Hypothesis}

\theoremstyle{definition}

\newtheorem*{theorem*}{Theorem}
\theoremstyle{remark}

\newtheorem*{ex*}{Example}
\theoremstyle{remark}
\newtheorem*{exs*}{Examples}
\theoremstyle{remark}
\newtheorem*{rems*}{Remarks}
\theoremstyle{remark}
\newtheorem*{rem*}{Remark}
\newtheorem{rem}[theo]{Remark}
\theoremstyle{definition}
\newtheorem*{idea*}{Idea}

\theoremstyle{remark}

\theoremstyle{remark}

\title[Nonrelativistic Limit of Generalized MIT Bag Models]{Nonrelativistic Limit of Generalized MIT Bag Models and Spectral Inequalities}

\author{Jussi~Behrndt}
\author{Dale~Frymark}
\author{Markus~Holzmann}
\author{Christian~Stelzer-Landauer}

\address{Institut f\"{u}r Angewandte Mathematik,
Technische Universit\"{a}t Graz,
Steyrergasse 30, A 8010 Graz, Austria.}

\email{behrndt@tugraz.at, dfrymark2@gmail.com, \newline holzmann@math.tugraz.at, christian.stelzer@tugraz.at}


\begin{document}

\begin{abstract}
For a family of self-adjoint Dirac operators $-i c (\alpha \cdot \nabla)  + \frac{c^2}{2}$ subject to generalized 
MIT bag boundary conditions on domains in $\mathbb R^3$, it is shown that the nonrelativistic limit in the norm resolvent sense is the Dirichlet Laplacian.
This allows for the transfer of spectral geometry results for Dirichlet Laplacians to Dirac operators for large $c$.  
\end{abstract}

\maketitle

\setcounter{tocdepth}{1}

\section{Introduction}

The MIT bag operator and more general types of self-adjoint Dirac operators on domains $\Omega\subset\mathbb R^3$ have attracted a lot of attention in the last years.
The MIT bag model itself originates from the investigation of quarks in hadrons from the 1970s \cite{B68,CJJTW74,DJJK75,J75} and has been studied from a more mathematical perspective in \cite{ALTR17,ALT20,ALMT19,BHM20,LO23,MOP20,OV16,R20,R21}. 
The present paper is inspired by the recent contribution \cite{AMSV22}, where spectral properties of the family $H_\kappa^{\Omega}$, $\kappa\in\mathbb R$, of self-adjoint Dirac operators 
\begin{equation} \label{def_H_tau}
  \begin{split}
    H_\kappa^{\Omega} f &= -i c (\alpha \cdot \nabla) f + \frac{c^2}{2} \beta f, \\
    \dom H_\kappa^{\Omega} &= \big\{ f \in H^1(\Omega; \mathbb{C}^4): f = i (\sinh (\kappa) I_4 - \cosh (\kappa) \beta) (\alpha \cdot \nu) f\,\,\text{on}\,\,\partial\Omega \big\},
  \end{split}
\end{equation}
in $L^2(\Omega; \mathbb{C}^4)$ were studied. Here $\alpha \cdot \nabla=\alpha_1 \partial_1 + \alpha_2 \partial_2 + \alpha_3 \partial_3$ with the usual 
Dirac matrices $\alpha_1, \alpha_2, \alpha_3, \beta \in \mathbb{C}^{4 \times 4}$ (see \eqref{e-diracmatrices} and \eqref{notation} below), $c>0$ is the speed of light, $\Omega$ is a $C^2$-domain with unit normal vector $\nu$, and 
$H^1(\Omega; \mathbb{C}^4)$ is the first order $L^2$-based Sobolev space. The operators $H_\kappa^{\Omega}$ model the propagation of a relativistic spin $\frac{1}{2}$
particle with mass $m = \frac{1}{2}$ subject to the boundary conditions 
in \eqref{def_H_tau}, which are a three-dimensional counterpart of the quantum dot boundary conditions; cf. \cite{BFSB17_1,BFSB17_2},
the introduction in \cite{AMSV22} for more references in dimension two, and Section~\ref{section_confinement} for a further motivation of the boundary conditions. In particular, for $\kappa = 0$ the standard 
MIT bag boundary conditions are recovered.
If $\Omega$ is bounded, then the spectrum of $H_\kappa^{\Omega}$ is purely discrete and consists of eigenvalues
\begin{equation}\label{evs}
  \dots \leq \lambda_2^{-}(H_\kappa^{\Omega}) \leq \lambda_1^{-}(H_\kappa^{\Omega}) \leq -\frac{c^2}{2} < \frac{c^2}{2} \leq \lambda_1^{+}(H_\kappa^{\Omega}) \leq \lambda_2^{+}(H_\kappa^{\Omega}) \leq \dots,
\end{equation}
that accumulate at $\pm\infty$. The main objective in \cite{AMSV22} is the analysis of the eigenvalue curves $\kappa\mapsto \lambda_j^\pm(H_\kappa^{\Omega})$ and their asymptotic behaviour, which
then leads to spectral geometry results for $H_\kappa^{\Omega}$ with $\kappa$ sufficiently large. The most remarkable result therein is a variant of the Faber-Krahn inequality for $\kappa$ sufficiently large
minimizing the first positive eigenvalue when $\Omega$ is a ball. For related spectral geometry results for two-dimensional Dirac operators with infinite mass boundary conditions we refer to \cite{ABLO21, BFSB17_2, BK22, LO19, V23}.

In this paper we propose a different approach to obtain spectral inequalities and spectral geometry results for the Dirac operators $H_\kappa^{\Omega}$, which is based on the analysis of a 
nonrelativistic limit. This allows us to conclude for all sufficiently large $c$ and \textit{all} $\kappa\in\mathbb R$,
e.g., the Faber-Krahn inequality for the first two positive eigenvalues $\lambda_1^{+}(H_\kappa^{\Omega})$, $\lambda_2^{+}(H_\kappa^{\Omega})$, the Hong-Krahn-Szeg\"o inequality minimizing the second two positive eigenvalues $\lambda_3^{+}(H_\kappa^{\Omega})$, $\lambda_4^{+}(H_\kappa^{\Omega})$, 
or the Payne-P\'{o}lya-Weinberger inequality for the ratios $\lambda_j^+(H_\kappa^{\Omega})/\lambda_l^+(H_\kappa^{\Omega})$, $j=1,2$, $l=3,4$, of the first two and the second two positive eigenvalues, relying on classical counterparts for the Dirichlet Laplacian \cite{AB92,F23,H54,K25,K26,P55}; here the spectral inequalities come for pairs of eigenvalues, as all eigenvalues of $H_\kappa^\Omega$ have even multiplicity, and remain valid in an analogous form also for the first two pairs of negative eigenvalues, see Remark~\ref{laterrem}.
In the same spirit other results from spectral geometry 
can be transferred from Laplacians to Dirac operators; we refer the reader to the monographs \cite{H06,LMP23,PS51} for an introduction to and overview of this topic, but limit ourselves to the above-mentioned three examples.

The nonrelativistic limit provides a connection of the generalized MIT bag models with their nonrelativistic counterparts, i.e. Schr\"odinger operators, and is of independent interest, as it gives a physical interpretation of $H_\kappa^\Omega$.
To find it one has to subtract the energy of the resting particle $\frac{c^2}{2}$ and compute the limit of the resolvent of $H_\kappa^{\Omega} - \frac{c^2}{2}$ as $c \rightarrow \infty$. 
Nonrelativistic limits of Dirac operators have been computed in many different settings. More information on three dimensional Dirac operators with regular potentials, for example, can be found in \cite[Chapter~6]{T92} and the references therein. In \cite{C14} it was shown that the nonrelativistic limit of a family of one-dimensional Dirac operators with boundary conditions containing the counterpart of $H_\kappa^\Omega$ is a Dirichlet or a Neumann Laplacian. Moreover, the nonrelativistic limit of one-dimensional Dirac operators with singular interactions supported on points, which are closely related to one-dimensional Dirac operators with boundary conditions, was studied extensively in \cite{BMP17, CMP13, GS87, HT22}. In higher dimensions, the nonrelativistic limit of Dirac operators with singular potentials supported on curves in $\mathbb{R}^2$ and surfaces in $\mathbb{R}^3$ was computed in various situations in \cite{BEHL18,BEHL19,BEHT23,BHS22}. We also point out the paper \cite{ALTR17}, where it is shown that for bounded $\Omega$ the discrete eigenvalues of the MIT bag model, i.e. of $H_\kappa^\Omega$ in~\eqref{def_H_tau} for $\kappa = 0$,  converge in the nonrelativistic limit to the eigenvalues of the Dirichlet Laplacian. However, in \cite{ALTR17} only the convergence of the eigenvalues and not of the operator itself was studied. 

In order to state our main result on the nonrelativistic limit of the operators $H_\kappa^\Omega$ we make the following assumption on $\Omega$, where we use the definition of a $C^2$-domain as, e.g., in \cite{M00}.

\begin{hypothesis} \label{hypothesis_Omega}
Let $\Omega\subset \mathbb{R}^3$ be a (bounded or unbounded) $C^2$-domain, not necessarily connected, with a compact boundary
and unit normal vector field $\nu$ pointing outwards of $\Omega$.
The bounded element in $\{\Omega,\mathbb{R}^3\setminus\overline\Omega\}$ is denoted by $\Omega_+$, 
the unbounded element in $\{\Omega,\mathbb{R}^3\setminus\overline\Omega\}$ is denoted by $\Omega_-$, and  $\nu_+$ is the unit normal vector field pointing outwards of $\Omega_+$,
so that  $\nu=\nu_+$ if $\Omega=\Omega_+$ and $\nu=-\nu_+$ if $\Omega=\Omega_-$.
For the common boundary we write
$\Sigma:=\partial \Omega = \partial \Omega_+=\partial \Omega_-$.
\end{hypothesis}

Then, the main result of the present paper reads as follows:

\begin{theo} \label{theorem_nr_limit}
  Let $\kappa \in \mathbb{R}$, $\Omega \subset \mathbb{R}^3$ be as in Hypothesis~\ref{hypothesis_Omega}, and $z \in \mathbb{C} \setminus [0, \infty)$. Then, 
  there exists a constant $K(z)$ such that 
  for all $c$ sufficiently large $z+\tfrac{c^2}{2} \in \rho(H_\kappa^{\Omega}) \cap \rho(-\Delta_D^{\Omega})$ and
  \begin{equation} \label{equation_nr_limit}
    \left\| \left( H_\kappa^{\Omega} - \left(z + \frac{c^2}{2}\right) \right)^{-1} - (-\Delta_D^{\Omega} - z)^{-1} \begin{pmatrix} I_2 & 0 \\ 0 & 0 \end{pmatrix} \right\|_{L^2(\Omega; \mathbb{C}^4) \rightarrow L^2(\Omega; \mathbb{C}^4)} \leq \frac{K(z)}{\sqrt{c}},
  \end{equation}
  where $-\Delta_D^{\Omega}$ denotes the self-adjoint Dirichlet Laplacian in $L^2(\Omega; \mathbb{C})$.
\end{theo}

The strategy to prove Theorem~\ref{theorem_nr_limit} is to consider the self-adjoint orthogonal sum $H_\kappa^{\Omega_+} \oplus H_\kappa^{\Omega_-}$ in $L^2(\Omega_+; \mathbb{C}^4)\oplus
L^2(\Omega_-; \mathbb{C}^4)$, which can be identified with
a self-adjoint Dirac operator $A_\kappa^{\Sigma}$ in $L^2(\mathbb R^3; \mathbb{C}^4)$ with a $\delta$-shell potential supported on $\Sigma$, see \cite{AMV15, BEHL19, BHM20, B22}.
Such types of Dirac operators with singular interactions are well-studied, see the review article \cite{BHSS22} and the references therein.
We collect some properties of $A_\kappa^{\Sigma}$ in Section~\ref{section_confinement} and provide a Krein type formula in Proposition~\ref{proposition_delta_op}
for its resolvent, which is the key tool for
the analysis of the nonrelativistic limit. Each of the terms appearing in the resolvent formula will be examined separately and the main technical difficulty
is the limit behavior of the inverse of 
\begin{equation}\label{schwierig}
\vartheta_c + \mathcal{M}_c \mathcal{C}_{z + c^2/2} \mathcal{M}_c,
\end{equation}
involving a strongly singular boundary integral operator $\mathcal{C}_{z + c^2/2}$ on $\Sigma$,
a coefficient matrix $\vartheta_c$ modelling the boundary condition in $\dom H_\kappa^\Omega$, and a scaling matrix $\mathcal{M}_c$ (see \eqref{e-cz},
\eqref{def_coefficients}, and \eqref{def_M_c} for details). In fact, it turns out that the  
operator in~\eqref{schwierig} does not converge to a boundedly invertible operator in one Sobolev space on $\Sigma$, but instead it is necessary 
to study the convergence of the inverse of \eqref{schwierig}
as an operator acting between different fractional order Sobolev spaces on $\Sigma$.
Here we argue via the Schur complement and rely on an advanced and deep analysis
of various boundary integral operators appearing in this context.
Eventually, it turns out that the limit of
$A_\kappa^{\Sigma}$ in the norm resolvent sense is an orthogonal sum of Dirichlet Laplacians and compressing the resolvents onto the original domain
leads to \eqref{equation_nr_limit}.

It is well-known that the operator norm convergence in \eqref{equation_nr_limit} implies the convergence of the corresponding spectra (see, e.g. \cite{kato,RS72,W00}) and, in particular, 
if $\Omega$ is bounded, the spectrum of $H_\kappa^\Omega$ is discrete and we conclude convergence of eigenvalues. This leads to spectral inequalities for the positive eigenvalues of 
the Dirac operators $H_\kappa^{\Omega}$, $\kappa\in\mathbb R$, for $c>0$ sufficiently large; cf. Remark~\ref{laterrem} for analogous results for the negative eigenvalues.

\begin{cor} \label{cor_Faber_Krahn}
  Let $\kappa \in \mathbb{R}$, $\Omega \subset \mathbb{R}^3$ be a bounded $C^2$-domain, $B \subset \mathbb{R}^3$ be a ball such that $\vert B\vert=\vert\Omega\vert$ and $B_1,B_2 \subset \R^3$ be identical and disjoint balls such that $\vert B_1\vert + \vert B_2\vert=\vert\Omega\vert$. Then, the following assertions hold for $c>0$ sufficiently large:
  \begin{itemize}
  	\item[(i)] $\lambda_{j}^+(H_\kappa^{B}) \leq \lambda_{j}^+(H_\kappa^{\Omega})$ for $j \in \{1,2\}$ and equality holds if and only if $\Omega$ is a ball.
  	\item[(ii)] $\lambda_{j}^+(H_\kappa^{B_1 \cup B_2}) \leq \lambda_{j}^+(H_\kappa^{\Omega})$ for $j \in \{3,4\}$ and equality holds if and only if $\Omega$ is the union of two identical disjoint balls.
  	\item[(iii)] If, in addition, $\Omega$ is connected, then 
  	\begin{equation*}
  		\frac{\lambda_{j}^+(H_\kappa^{B})}{\lambda_{l}^+(H_\kappa^{B})} \leq \frac{\lambda_{j}^+(H_\kappa^{\Omega})}{ \lambda_{l}^+(H_\kappa^{\Omega})}, \qquad j \in \{1,2\},\, l \in \{3,4\},
  	\end{equation*}
   and  equality holds if and only if $\Omega$ is a ball.
  \end{itemize}
\end{cor}

The article is organized as follows. In Section~\ref{s-prelims} we introduce the free Dirac operator in $\mathbb{R}^3$ and some associated integral operators, show the connection of $H_\kappa^\Omega$ and Dirac operators $A_\kappa^{\Sigma}$ with singular interactions supported on $\Sigma = \partial \Omega$, and recall some properties of the Dirichlet Laplacian. In Section~\ref{s-nonrel} we compute the nonrelativistic limit of $A_\kappa^{\Sigma}$, which allows us to prove Theorem~\ref{theorem_nr_limit} and Corollary~\ref{cor_Faber_Krahn}.

\subsection*{Notations}

The Dirac matrices are denoted by 
\begin{align}\label{e-diracmatrices}
\al_k=\begin{pmatrix}
    0 & \sigma_k \\
    \sigma_k & 0
\end{pmatrix},\quad k\in\{1, 2, 3\},\quad\quad 
\beta=\begin{pmatrix}
    I_2 & 0 \\
    0 & -I_2
\end{pmatrix},
\end{align}
where $I_n$ is the $n\times n$ identity matrix, $n \in \mathbb{N}$, and 
\begin{align*}
\sigma_1=\begin{pmatrix}
0 & 1 \\
1 & 0
\end{pmatrix},
\hspace{2em}
\sigma_2=\begin{pmatrix}
    0 & -i \\
    i & 0
    \end{pmatrix},
    \hspace{2em}
    \sigma_3=\begin{pmatrix}
        1 & 0 \\
        0 & -1
    \end{pmatrix},
\end{align*}
are the Pauli matrices.
The Dirac matrices satisfy 
\begin{equation} \label{anti_commutation}
  \alpha_j \alpha_k + \alpha_k \alpha_j = 2 \delta_{j k} I_4, \quad \alpha_j \beta + \beta \alpha_j = 0, \quad j, k \in \{ 1, 2, 3 \},
\end{equation}
where $\delta_{j k}$ is the Kronecker symbol.
Moreover,  the notations
\begin{equation}\label{notation}
  \alpha \cdot \nabla = \sum_{j=1}^3 \alpha_j \partial_j \quad \text{and} \quad \alpha \cdot x = \sum_{j=1}^3 \alpha_j x_j, \quad x = (x_1, x_2, x_3) \in \mathbb{C}^3,
\end{equation}
will often be used.

If $M \subset \mathbb{R}^3$ and $k,l \in \mathbb{N}$, then the set of all continuous and $k$ times continuously differentiable functions $f: M \rightarrow \mathbb{C}^l$ is denoted by $C(M; \mathbb{C}^l)$ and $C^k(M; \mathbb{C}^l)$, respectively. Next,
denote by $\mathcal{F}$ the Fourier transform on the space $\mathcal{S}'(\R^3;\C)$ of tempered distributions. For the Sobolev spaces 
$H^s(\RR^3; \mathbb{C})$, $s \in \R $, we shall use the definition  
\begin{align}\label{e-fouriersobolev}
	H^s(\RR^3; \mathbb{C})=\left\{f \in \mathcal{S}'(\R^3;\C)~:~\int_{\RR^3}(1+|x|^2)^s|\cF f(x)|^2 dx<\infty\right\},
\end{align}
with Hilbert space norm
\begin{equation} \label{norm_Sobolev_space}
	\| f \|_{H^s(\mathbb{R}^3; \mathbb{C})}^2 := \int_{\RR^3}(1+|x|^2)^s |\cF f(x)|^2 dx, \quad f \in H^s(\mathbb{R}^3; \mathbb{C}).
\end{equation}
For $\Omega$ as in Hypothesis~\ref{hypothesis_Omega} the Sobolev spaces $H^s(\Omega;\C)$ are defined via restrictions of functions from $H^s(\RR^3; \mathbb{C})$
onto $\Omega$, and the spaces $H^t(\Sigma;\C)$, $t \in [-2,2]$, on the boundary $\Sigma$ of $\Omega$ are defined by using an open cover of $\Sigma$ and a corresponding 
partition of unity, reducing it to Sobolev spaces on hypographs; see, e.g., \cite[Chapter 3]{M00} for more details. 
We denote by $\gamma_D: H^1(\Omega; \mathbb{C}) \rightarrow H^{1/2}(\Sigma; \mathbb{C})$ the bounded Dirichlet trace operator and we shall use the same symbol for the trace operator $\gamma_D: H^1(\mathbb{R}^3; \mathbb{C}) \rightarrow H^{1/2}( \Sigma; \mathbb{C})$. 
Sobolev spaces of vector valued functions are defined component-wise and in this
context the action of the Dirichlet trace operator is also understood component-wise.

If $A$ is a linear operator acting between two Hilbert spaces $\mathcal{H}$ and $\mathcal{G}$, then its domain, range, and kernel are denoted by $\dom A$, $\ran A$, and $\ker A$, respectively. Whenever $A$ is bounded and everywhere defined, then $\| A \|_{\mathcal{H} \rightarrow \mathcal{G}}$ is the operator norm of $A$. If $A$ is self-adjoint in $\mathcal{H}$, then the symbols $\rho(A)$, $\sigma(A)$, $\sigma_{\textup{ess}}(A)$, and $\sigma_{\textup{disc}}(A)$ are used for the resolvent set, spectrum, essential spectrum, and discrete spectrum of $A$, respectively.

\section{Preliminaries}\label{s-prelims}

In this preliminary section we first collect several results about the free Dirac operator in $\mathbb{R}^3$ and associated integral operators. Afterwards,
we show how the operators $H_\kappa^{\Omega}$ in~\eqref{def_H_tau} are related to Dirac operators with $\delta$-shell potentials,
and we recall some useful properties of the single layer potential, single layer boundary integral operator, and the Dirichlet Laplacian 
that are needed to prove Theorem~\ref{theorem_nr_limit}.
Throughout this section we assume that $\Omega$, $\Omega_\pm$, and $\Sigma$ are as in Hypothesis~\ref{hypothesis_Omega}.

\subsection{The free Dirac operator and associated integral operators} \label{section_free_op}
It is well-known that the free Dirac operator 
\begin{align}\label{e-freedirac}
A_0 f= -i c (\alpha \cdot \nabla) f + \frac{c^2}{2} \beta f, \hspace{2em} \dom A_0 = H^1(\RR^3;\CC^4),
\end{align}
in $\mathbb{R}^3$ is self-adjoint in $L^2(\RR^3;\CC^4)$ and its spectrum is $\sigma(A_0)=(-\infty,-\frac{c^2}{2}]\cup[\frac{c^2}{2},\infty)$. 
For $z\in\rho(A_0)=\CC\setminus ((-\infty,-\frac{c^2}{2}]\cup[\frac{c^2}{2},\infty))$ and $f\in L^2(\RR^3;\CC^4)$, the resolvent of $A_0$ is given by
\begin{align*}
    (A_0-z)^{-1}f(x)=\int_{\RR^3}G_z(x-y)f(y)dy, \hspace{2em} x\in\RR^3,
\end{align*}
where the function $G_z: \mathbb{R}^3 \setminus \{ 0 \} \rightarrow \mathbb{C}^{4 \times 4}$ is defined by
\begin{align}\label{e-greens}
    G_z(x)=\left(\frac{z}{c^2}I_4+\frac{1}{2}\beta+\left(1-i\sqrt{\frac{z^2}{c^2}-\frac{c^2}{4}}|x|\right)\frac{i(\al\cdot x)}{c|x|^2}\right)\frac{e^{i\sqrt{z^2/c^2-c^2/4}|x|}}{4\pi|x|}
\end{align}
and the square root is chosen such that $\textup{Im}\,\sqrt{z^2/c^2-c^2/4}>0$; cf. \cite[Section~1.E]{T92}.

Next, we introduce several integral operators and summarize some of their properties that are necessary to prove Theorem~\ref{theorem_nr_limit}; we refer to \cite{BEHL19, BH20, BHSS22} for more details. In the following  $\gamma_D: H^1(\mathbb{R}^3; \mathbb{C}^4) \rightarrow H^{1/2}(\Sigma; \mathbb{C}^4)$ denotes the Dirichlet trace operator. 
For $z\in\rho(A_0)$ the map 
\begin{equation} \label{def_Phi_z_star}
  \Phi_z^* := \gamma_D (A_0 - \overline{z})^{-1}: L^2(\mathbb{R}^3; \mathbb{C}^4) \rightarrow H^{1/2}(\Sigma; \mathbb{C}^4)
\end{equation}
is well-defined and bounded. It is not difficult to see that $\Phi_z^*$ acts on $f \in L^2(\mathbb{R}^3; \mathbb{C}^4)$ as
\begin{equation*}
  \Phi_z^* f(x) = \int_{\mathbb{R}^3} G_{\overline{z}}(x-y) f(y) dy, \quad x \in \Sigma.
\end{equation*}
The definition of $\Phi_z^*$ in~\eqref{def_Phi_z_star} allows to define the bounded anti-dual map
\begin{equation} \label{def_Phi_z}
  \Phi_z := ( \Phi_z^* )': H^{-1/2}(\Sigma; \mathbb{C}^4) \rightarrow L^2(\mathbb{R}^3; \mathbb{C}^4).
\end{equation}
With the help of Fubini's theorem and $(G_{\overline{z}}(x))^* = G_z(-x)$ one shows that $\Phi_z$ acts on $\varphi \in L^2(\Sigma; \mathbb{C}^4)$ as
\begin{align*}
    \Phi_z \f (x):=\int_\Sigma G_z(x-y)\f(y)d\sigma(y), \hspace{2em} x\in\RR^3\setminus\Sigma,
\end{align*}
where $d\sigma$ denotes the surface measure on $\Sigma$. 
We will also make use of the strongly singular boundary integral operator $\cC_z:L^2(\Sigma;\CC^4)\to L^2(\Sigma;\CC^4)$, $z\in\rho(A_0)$, acting via
\begin{align}\label{e-cz}
    \cC_z \f (x):=\lim_{\eps\to0^+}\int_{\Sigma\setminus B(x,\eps)}G_z(x-y)\f(y)d\sigma(y), \hspace{2em} x\in\Sigma,~\varphi \in L^2(\Sigma; \mathbb{C}^4),
\end{align}
where $B(x,\eps)$ is the ball of radius $\eps$ centered at $x$. For $s \in [0, \frac{1}{2}]$ the map $\mathcal{C}_z$ gives rise to a bounded operator
\begin{equation} \label{C_z_mapping_properties}
  \mathcal{C}_z: H^s(\Sigma; \mathbb{C}^4) \rightarrow H^s(\Sigma; \mathbb{C}^4).
\end{equation}
The adjoint of the realization of $\mathcal{C}_z$ in $L^2(\Sigma; \mathbb{C}^4)$ satisfies $\mathcal{C}_z^* = \mathcal{C}_{\overline{z}}$ and 
it follows from~\eqref{C_z_mapping_properties} that $\mathcal{C}_z$ admits a bounded extension to $H^s(\Sigma; \mathbb{C}^4)$, $s \in [-\frac{1}{2}, 0]$, 
such that
\begin{equation} \label{C_z_dual}
  \mathcal{C}_z = (\mathcal{C}_{\overline{z}})': H^s(\Sigma; \mathbb{C}^4) \rightarrow H^s(\Sigma; \mathbb{C}^4), \quad s \in \big[-\tfrac{1}{2}, 0 \big],
\end{equation}
where $(\mathcal{C}_{\overline{z}})'$ denotes the anti-dual of $\mathcal{C}_{\overline{z}}$.

\subsection{$H_\kappa^{\Omega}$ and Dirac operators with $\delta$-shell potentials} \label{section_confinement}

In this subsection we show how the operators $H_\kappa^{\Omega}$ defined in~\eqref{def_H_tau} are related to Dirac operators $A_\kappa^{\Sigma}$ with $\delta$-shell 
potentials supported on $\Sigma$; the latter operators are well-studied, see, e.g., \cite{AMV15, BEHL19, BHSS22} and the references therein.
Recall the notation $\Omega_\pm$ and the unit outward normal vector field $\nu_+$ from Hypothesis~\ref{hypothesis_Omega}. For a function $f: \mathbb{R}^3 \rightarrow \mathbb{C}^4$ we
write $f_\pm := f \upharpoonright \Omega_\pm$. 
Define the operator
\begin{equation} \label{def_delta_op}
  \begin{split}
    A_\kappa^\Sigma &= \big( -i c (\alpha \cdot \nabla) + \tfrac{c^2}{2} \beta \big) f_+ \oplus \big( -i c (\alpha \cdot \nabla) + \tfrac{c^2}{2} \beta \big) f_-, \\
    \dom A_\kappa^\Sigma &= \big\{ f = f_+ \oplus f_- \in H^1(\Omega_+; \mathbb{C}^4) \oplus H^1(\Omega_-; \mathbb{C}^4): \\
    &\qquad \qquad -i (\alpha \cdot \nu_+) (\gamma_D f_+ - \gamma_D f_-) = (\sinh (\kappa) I_4 + \cosh (\kappa) \beta) (\gamma_D f_+ + \gamma_D f_-) \big\},
  \end{split}
\end{equation}
in $L^2(\mathbb{R}^3; \mathbb{C}^4)$.
We note that $A_\kappa^\Sigma$ is  the rigorously defined operator associated with the formal differential expression 
$-i c (\alpha \cdot \nabla) + \frac{c^2}{2} \beta + 2 c (\sinh (\kappa) I_4 + \cosh (\kappa) \beta) \delta_\Sigma$.

Our first observation is an immediate consequence from \cite[Lemma~3.1~(ii)]{BEHL19}, which says that the operator formally given by $-i c (\alpha \cdot \nabla) + \frac{c^2}{2} \beta + (\eta I_4 + \tau \beta) \delta_\Sigma$ decouples to the orthogonal sum of two Dirac operators with boundary conditions acting on functions in $\Omega_\pm$ if and only if  $\eta^2 - \tau^2 = -4c^2$; in the present setting 
the strength $\eta$ of the electrostatic interaction in \cite{BEHL19} is $2 c \sinh (\kappa)$, the strength $\tau$ of the Lorentz scalar interaction is $2c\cosh (\kappa)$, and the normal vector in the definition of $H_\kappa^{\Omega_-}$ in~\eqref{def_H_tau} is $-\nu_+$.
Note that this choice of $\eta$ and $\tau$ is a natural parametrization of the arm of the hyperbola $\eta^2 - \tau^2 = -4c^2$ that contains the MIT bag boundary conditions.
We also refer the reader to  \cite[Section~5]{AMV15},  \cite[Section~5.3]{BHM20}, or \cite[Section~5.2]{BHSS22} for similar statements.

\begin{lem} \label{lemma_confinement}
  The equality $A_\kappa^\Sigma = H_\kappa^{\Omega_+} \oplus H_\kappa^{\Omega_-}$ holds.
\end{lem}

In the next proposition we summarize some properties of the operator $A_\kappa^\Sigma$  that will be particularly useful for our analysis. Recall that $A_0$ is the free Dirac operator defined in~\eqref{e-freedirac} and that $\Phi_z$ and $\mathcal{C}_z$ are the operators defined in~\eqref{def_Phi_z} and~\eqref{e-cz}, respectively. Moreover, define  the two numbers
\begin{equation} \label{def_coefficients}
  a_+ := \frac{1}{2} (\cosh (\kappa) - \sinh (\kappa)) > 0, \quad a_- := -\frac{1}{2} (\cosh (\kappa) + \sinh (\kappa)) < 0,
\end{equation}
and for $c > 0$ the coefficient matrix $\vartheta_c$ and the scaling matrix $\mathcal{M}_c$
\begin{equation} \label{def_M_c}
  \vartheta_c := \begin{pmatrix} \tfrac{1}{c} a_+ I_2 & 0 \\ 0 & a_- I_2 \end{pmatrix}  \in \mathbb{C}^{4 \times 4}, \qquad 
  \mathcal{M}_c := \begin{pmatrix} I_2 & 0 \\ 0 & \sqrt{c} I_2 \end{pmatrix}  \in \mathbb{C}^{4 \times 4}.
\end{equation}

\begin{prop} \label{proposition_delta_op}
  Let $\kappa \in \mathbb{R}$ and $c > 0$. Then, the operator $A_\kappa^\Sigma$ in~\eqref{def_delta_op} is self-adjoint in $L^2(\mathbb{R}^3; \mathbb{C}^4)$, $\sigma(A_\kappa^\Sigma) = (-\infty, -\frac{c^2}{2}] \cup [\frac{c^2}{2}, \infty)$, for $z \in \rho(A_\kappa^\Sigma)$ the operator $\vartheta_c + \mathcal{M}_c \mathcal{C}_z \mathcal{M}_c$ admits a bounded inverse in $H^{1/2}(\Sigma; \mathbb{C}^4)$, and the formula
  \begin{equation*}
    (A_\kappa^\Sigma - z)^{-1} = (A_0 - z)^{-1} - \Phi_z \mathcal{M}_c \big( \vartheta_c + \mathcal{M}_c \mathcal{C}_z \mathcal{M}_c \big)^{-1} \mathcal{M}_c \Phi_{\overline{z}}^*
  \end{equation*}
  holds.
\end{prop}
\begin{proof} 
  It follows from \cite[Lemma~3.3 and Theorems~3.4 \&~4.1]{BEHL19} or \cite[Theorem~5.6]{BHSS22} (in the case $c=1$) that $A_\kappa^\Sigma$ is self-adjoint in $L^2(\mathbb{R}^3; \mathbb{C}^4)$, that $\sigma_{\textup{ess}}(A_\kappa^\Sigma) = (-\infty, -\frac{c^2}{2}] \cup [\frac{c^2}{2}, \infty)$, that $I_4 + 2 c (\sinh (\kappa) I_4 + \cosh (\kappa) \beta) \mathcal{C}_z$ is bijective in $H^{1/2}(\Sigma; \mathbb{C}^4)$ for $z \in \rho(A_\kappa^\Sigma)\cap\rho(A_0)$ and that the  resolvent formula
  \begin{equation*}
    (A_\kappa^\Sigma - z)^{-1} = (A_0 - z)^{-1} - \Phi_z \big( I_4 + 2c  (\sinh (\kappa) I_4 + \cosh (\kappa) \beta) \mathcal{C}_z \big)^{-1} 2c  (\sinh (\kappa) I_4 + \cosh (\kappa) \beta) \Phi_{\overline{z}}^*
  \end{equation*}
  holds. Note that the matrix $2c  (\sinh (\kappa) I_4 + \cosh (\kappa) \beta)$ is invertible with inverse 
  \begin{equation*}
    \big( 2c  (\sinh (\kappa) I_4 + \cosh (\kappa) \beta) \big)^{-1} = \frac{1}{2 c} (-\sinh (\kappa) I_4 + \cosh (\kappa) \beta) = \mathcal{M}_c^{-1} \vartheta_c \mathcal{M}_c^{-1}.
  \end{equation*}
  Hence, also $\vartheta_c + \mathcal{M}_c \mathcal{C}_z \mathcal{M}_c = \mathcal{M}_c (\frac{1}{2 c} (-\sinh (\kappa) I_4 + \cosh (\kappa) \beta) + \mathcal{C}_z) \mathcal{M}_c$ is bijective in $H^{1/2}(\Sigma; \mathbb{C}^4)$ and the claimed resolvent formula is true.

  It remains to show that $(-\frac{c^2}{2}, \frac{c^2}{2}) \cap \sigma(A_\kappa^\Sigma) =(-\frac{c^2}{2}, \frac{c^2}{2}) \cap \sigma_\textup{p}(A_\kappa^\Sigma) = \emptyset$. For this, we use the Birman-Schwinger principle for $A_\kappa^\Sigma$ from \cite[Lemma~3.3]{BEHL19}, which states that
  \begin{equation} \label{Birman_Schwinger}
    z \in \left(-\frac{c^2}{2}, \frac{c^2}{2}\right) \cap \sigma_{\textup{p}}(A_\kappa^\Sigma) \quad \text{if and only if} \quad 0 \in \sigma_{\textup{p}}(I_4 + 2 c (\sinh (\kappa) I_4 + \cosh (\kappa) \beta) \mathcal{C}_z).
  \end{equation}
  Let $z \in (-\frac{c^2}{2}, \frac{c^2}{2})$ and assume that $\varphi \in \ker (I_4 + 2 c (\sinh (\kappa) I_4 + \cosh (\kappa) \beta) \mathcal{C}_z)$. Then, 
  \begin{equation}\label{jadoch22}
    \begin{split}
      0 &= \big( (I_4 + 2 c  \mathcal{C}_z (-\sinh (\kappa) I_4 + \cosh (\kappa) \beta)) (I_4 + 2 c (\sinh (\kappa) I_4 + \cosh (\kappa) \beta) \mathcal{C}_z) \varphi, \varphi \big)_{L^2(\Sigma; \mathbb{C}^4)} \\
      &= \left( (I_4 + 4 c^2  \mathcal{C}_z^2 + 2 c \cosh (\kappa)  (\mathcal{C}_z \beta  + \beta \mathcal{C}_z)) \varphi, \varphi\right)_{L^2(\Sigma; \mathbb{C}^4)}.
    \end{split}
  \end{equation}
  With~\eqref{anti_commutation} and \eqref{e-greens} one finds that
  $$\mathcal{C}_z \beta  + \beta \mathcal{C}_z = 2 \left(\frac{1}{2} I_4 + \frac{z}{c^2} \beta\right) \mathcal{S}_{z^2/c^2 - c^2/4},$$ 
  where $\mathcal{S}_{z^2/c^2 - c^2/4}$ is the single layer boundary integral operator defined below in \eqref{e-singlelayer}. 
  In the present situation we have $z^2/c^2 - c^2/4<0$ and hence it follows that $\mathcal{S}_{z^2/c^2 - c^2/4}$
  is a non-negative operator in $L^2(\Sigma; \mathbb{C})$; cf. the text below \eqref{mapping_properties_S_mu} in the next subsection. Therefore, 
  $I_4 + 4 c^2  \mathcal{C}_z^2 + 2 c \cosh (\kappa)  (\mathcal{C}_z \beta  + \beta \mathcal{C}_z)$ is a strictly positive operator in $L^2(\Sigma; \mathbb{C}^4)$ and 
  we obtain $\varphi = 0$ from \eqref{jadoch22}. Therefore, by~\eqref{Birman_Schwinger} we have $z \notin \sigma_{\textup{p}}(A_\kappa^\Sigma)$.
\end{proof}

From the properties of $A_\kappa^\Sigma$ one can now easily deduce the properties of $H_\kappa^{\Omega}$ stated in the following corollary, when $\Omega$ coincides either with $\Omega_+$ or $\Omega_-$. 
The claims follow immediately from Lemma~\ref{lemma_confinement} and Proposition~\ref{proposition_delta_op}; for~(i) one additionally uses that 
$\dom H_\kappa^{\Omega} \subset H^1(\Omega; \mathbb{C}^4)$ is compactly embedded in $L^2(\Omega; \mathbb{C}^4)$ if $\Omega$ is bounded, see also \cite[Lemma~1.2]{AMSV22}, and that $H_\kappa^{\Omega}$ commutes with the anti-linear time reversal operator $\mathcal{T} f = -i \big( \begin{smallmatrix} 0 & I_2 \\ I_2 & 0 \end{smallmatrix} \big) \alpha_2 \overline{f}$, see the proof of \cite[Proposition~4.2~(ii)]{BEHL19} for details.

\begin{cor} \label{corollary_H_tau_properties}
  Let $\kappa \in \mathbb{R}$ and $c > 0$. Then, the operator $H_\kappa^{\Omega}$ in~\eqref{def_H_tau} is self-adjoint in $L^2(\Omega; \mathbb{C}^4)$ and the following holds:
  \begin{itemize}
    \item[(i)] If $\Omega$ is bounded, then $\sigma(H_\kappa^{\Omega}) = \sigma_{\textup{disc}}(H_\kappa^{\Omega}) \subset (-\infty, -\frac{c^2}{2}] \cup[\frac{c^2}{2}, \infty)$ and all eigenvalues of $H_\kappa^\Omega$ have even multiplicity.
    \item[(ii)] If $\Omega$ is unbounded, then $\sigma(H_\kappa^{\Omega}) = (-\infty, -\frac{c^2}{2}] \cup[\frac{c^2}{2}, \infty)$.
  \end{itemize}  
  Moreover, for $z \in \mathbb{C} \setminus ((-\infty, -\frac{c^2}{2}] \cup [\frac{c^2}{2}, \infty))$ the resolvent formula
  \begin{equation*}
    \begin{split}
      (H_\kappa^{\Omega} - z)^{-1} &= P_{\Omega} (A_0 - z)^{-1} P_{\Omega}^*- P_{\Omega} \Phi_z \mathcal{M}_c \big( \vartheta_c + \mathcal{M}_c \mathcal{C}_z \mathcal{M}_c \big)^{-1} \mathcal{M}_c \Phi_{\overline{z}}^* P_{\Omega}^*
    \end{split}
  \end{equation*}
  holds, where $P_\Omega: L^2(\mathbb{R}^3; \mathbb{C}^4) \rightarrow L^2(\Omega; \mathbb{C}^4)$ is the projection operator acting as $f\mapsto  f \upharpoonright \Omega$ 
  and its adjoint $P_{\Omega}^*:L^2(\Omega; \mathbb{C}^4)\rightarrow L^2(\mathbb{R}^3; \mathbb{C}^4)$ is the embedding operator which extends a function $g \in L^2(\Omega; \mathbb{C}^4)$ by zero.
\end{cor}

\subsection{The Dirichlet Laplacian and associated integral operators} \label{section_Dirichlet_Laplacian}
We begin by briefly recalling some properties of the single layer potential and single layer boundary integral operator associated with $-\Delta-\mu$, where $-\Delta$ is the self-adjoint  
Laplacian in $L^2(\RR^3; \mathbb{C})$ defined on $H^2(\mathbb{R}^3; \mathbb{C})$ and $\mu\in\rho(-\Delta)=\CC\setminus[0,\infty)$. 

For $\f \in L^2(\Sigma; \CC)$ the single layer potential $SL_\mu$ is the formal integral operator that acts as
\begin{align}\label{e-singlelayer_potential}
SL_\mu \f (x)=\int_\Sigma \frac{e^{i\sqrt{\mu}|x-y|}}{4\pi|x-y|}\f(y)d\sigma(y), \hspace{2em} x\in \mathbb{R}^3 \setminus \Sigma,
\end{align}
and the single layer boundary integral operator $\mathcal{S}_\mu$ is the mapping defined by
\begin{align}\label{e-singlelayer}
\cS_\mu \f (x)=\int_\Sigma \frac{e^{i\sqrt{\mu}|x-y|}}{4\pi|x-y|}\f(y)d\sigma(y), \hspace{2em} x\in\Sigma,
\end{align}
where $\sqrt{\mu}$ is again the complex square root satisfying $\textup{Im}\,\sqrt{\mu}>0$ for $\mu\in\CC\setminus[0,\infty)$. 
It is well-known that for any $s \in [-\frac{1}{2}, \frac{1}{2}]$ the map $\mathcal{S}_\mu$ gives rise to a bounded and bijective operator
\begin{equation} \label{mapping_properties_S_mu}
  \mathcal{S}_\mu: H^s(\Sigma; \mathbb{C}) \rightarrow H^{s+1}(\Sigma; \mathbb{C}).
\end{equation}
Moreover, we will use that for $\mu < 0$ the realization of $\mathcal{S}_\mu$ in $L^2(\Sigma; \mathbb{C})$ is self-adjoint and non-negative. These claims can be shown in the same way as in \cite[Lemma~2.6]{BEHL20}, where
the two-dimensional case and $\mu=-1$ is treated.
Furthermore, by \cite[Corollary~6.14]{M00} the mapping $SL_\mu$ gives for any $s \in (-\frac{1}{2}, 1]$ rise to a bounded operator
\begin{equation*} 
  SL_\mu: H^{s-1/2}(\Sigma; \mathbb{C}) \rightarrow H^{s+1}(\Omega_+; \mathbb{C}) \oplus H^{s+1}(\Omega_-; \mathbb{C}).
\end{equation*}
Moreover, the representations
\begin{equation*} 
  SL_\mu = (-\Delta - \mu)^{-1} \gamma_D':H^{-1/2}(\Sigma; \mathbb{C})\rightarrow H^1(\mathbb{R}^3; \mathbb{C})
\end{equation*}
and
\begin{equation} \label{single_layer_BIO_representation_trace2}
  \mathcal{S}_\mu = \gamma_D (-\Delta - \mu)^{-1} \gamma_D':H^{-1/2}(\Sigma; \mathbb{C})\rightarrow H^{1/2}(\Sigma; \mathbb{C})
\end{equation}
hold, where 
$\gamma_D: H^1(\mathbb{R}^3; \mathbb{C}) \rightarrow H^{1/2}(\Sigma; \mathbb{C})$ is the bounded Dirichlet trace operator and $\gamma_D': H^{-1/2}(\Sigma; \mathbb{C}) \rightarrow H^{-1}(\mathbb{R}^3; \mathbb{C})$ its anti-dual map.
We will also use that the $L^2$-adjoint of $SL_\mu$ is given by
\begin{equation} \label{mapping_properties_SL_mu_star}
  SL_\mu^* = \gamma_D (-\Delta - \overline{\mu})^{-1}: L^2(\mathbb{R}^3; \mathbb{C}) \rightarrow H^{3/2}(\Sigma; \mathbb{C}),
\end{equation}
which is bounded, as the restriction $\gamma_D:H^2(\mathbb{R}^3; \mathbb{C})\rightarrow H^{3/2}(\Sigma; \mathbb{C})$ is bounded. Next, we state a useful continuity property of the map $\mu \mapsto \mathcal{S}_\mu$.

\begin{lem} \label{lemma_single_layer_continuous}
  Let $M \subset \mathbb{C} \setminus [0, \infty)$ be compact. Then, for all  $\mu_1, \mu_2 \in M$ the operator $\mathcal{S}_{\mu_1} - \mathcal{S}_{\mu_2}$ has a bounded extension from $H^{-3/2}(\Sigma; \mathbb{C})$ to $H^{3/2}(\Sigma; \mathbb{C})$ and there exists a constant $K(M)  > 0$ such that the estimate
  \begin{equation} \label{continuity_S_mu}
    \big\| \mathcal{S}_{\mu_1} - \mathcal{S}_{\mu_2} \big\|_{H^{-3/2}(\Sigma; \mathbb{C}) \rightarrow H^{3/2}(\Sigma; \mathbb{C})} \leq K(M) |\mu_1 - \mu_2|
  \end{equation}
  holds. In particular, for any $s \in [-\frac{3}{2}, \frac{3}{2}]$ the operator $\mathcal{S}_{\mu}: H^s(\Sigma; \mathbb{C}) \rightarrow H^s(\Sigma; \mathbb{C})$ is uniformly bounded in $\mu \in M$. 
\end{lem}
\begin{proof}
  It suffices to show that $(-\Delta - \mu_1)^{-1} - (-\Delta - \mu_2)^{-1} = (\mu_1 - \mu_2) (-\Delta - \mu_1)^{-1} (-\Delta - \mu_2)^{-1}$ gives rise to a bounded operator from $H^{-2}(\mathbb{R}^2; \mathbb{C})$ to $H^2(\mathbb{R}^2; \mathbb{C})$ that satisfies
  \begin{equation} \label{resolvent_Laplace_continuous}
    \big\| (-\Delta - \mu_1)^{-1} - (-\Delta - \mu_2)^{-1} \big\|_{H^{-2}(\mathbb{R}^3; \mathbb{C}) \rightarrow H^{2}(\mathbb{R}^3; \mathbb{C})} \leq K(M) |\mu_1 - \mu_2|,
  \end{equation}
  as then \eqref{continuity_S_mu} follows from~\eqref{single_layer_BIO_representation_trace2} and  the fact that $\gamma_D$ has a continuous restriction
  $\gamma_D: H^2(\mathbb{R}^3; \mathbb{C}) \rightarrow H^{3/2}(\Sigma; \mathbb{C})$ and $\gamma_D'$ a continuous extension $\gamma_D': H^{-3/2}(\Sigma; \mathbb{C}) \rightarrow H^{-2}(\mathbb{R}^3; \mathbb{C})$. To show~\eqref{resolvent_Laplace_continuous}, we compute for $f \in H^{-2}(\mathbb{R}^3; \mathbb{C})$, taking~\eqref{norm_Sobolev_space} into account,
  \begin{equation*}
    \begin{split}
      \big\| \big((-\Delta - \mu_1)^{-1} &- (-\Delta - \mu_2)^{-1}\big) f \big\|_{H^{2}(\mathbb{R}^3; \mathbb{C})}^2 \\
      & = \int_{\RR^3}(1+|x|^2)^2 \big|\cF \big((-\Delta - \mu_1)^{-1} - (-\Delta - \mu_2)^{-1}\big) f(x)\big|^2 dx \\
      & = \int_{\RR^3} (1+|x|^2)^2 \left| \frac{1}{|x|^2 - \mu_1} - \frac{1}{|x|^2 - \mu_2} \right|^2 |\cF f(x)|^2 dx \\
      & = \int_{\RR^3}  \frac{(1+|x|^2)^4 |\mu_1 - \mu_2|^2}{|(|x|^2 - \mu_1)(|x|^2 - \mu_2)|^2} \frac{|\cF f(x)|^2}{(1+|x|^2)^2} dx \\
      &\leq \sup_{x \in \mathbb{R}^3} \frac{(1+|x|^2)^4 |\mu_1 - \mu_2|^2}{|(|x|^2 - \mu_1)(|x|^2 - \mu_2)|^2} \cdot \| f \|_{H^{-2}(\mathbb{R}^3; \mathbb{C})}^2.
    \end{split}
  \end{equation*}
  This shows~\eqref{resolvent_Laplace_continuous} with $K(M) := \sup_{x \in \mathbb{R}^3, \mu_1, \mu_2 \in M} \frac{(1+|x|^2)^2}{|(|x|^2 - \mu_1)(|x|^2 - \mu_2)|}$.
  
  Eventually, it follows from~\eqref{continuity_S_mu} that $\mathcal{S}_{\mu}: H^{-3/2}(\Sigma; \mathbb{C}) \rightarrow H^{3/2}(\Sigma; \mathbb{C})$ is uniformly bounded in $\mu \in M$. 
  Since $H^{s_1}(\Sigma; \mathbb{C})$ is continuously embedded in $H^{s_2}(\Sigma; \mathbb{C})$ for $s_1 > s_2$, we conclude that 
  $\mathcal{S}_{\mu}: H^s(\Sigma; \mathbb{C}) \rightarrow H^s(\Sigma; \mathbb{C})$ is also uniformly bounded in $\mu \in M$ for any $s \in [-\frac{3}{2}, \frac{3}{2}]$. 
\end{proof}

Let again $\Omega$ be a $C^2$-domain as in  Hypothesis~\ref{hypothesis_Omega}.
In the next lemma we express the resolvent of the self-adjoint Dirichlet Laplacian 
\begin{equation} \label{def_Dirichlet_op}
  -\Delta_D^\Omega f = -\Delta f, \qquad \dom (-\Delta_D^\Omega) = \big\{ f \in H^2(\Omega; \mathbb{C}): \gamma_D f =  0 \big\},
\end{equation}
in $L^2(\Omega; \mathbb{C})$
as the compression of the resolvent of the self-adjoint Laplacian $-\Delta$ in $L^2(\mathbb{R}^3; \mathbb{C})$
and a perturbation term. 
The statement follows from, e.g., \cite[Theorem~4.4]{AB09}, \cite[Theorem~3.2]{B21} or \cite[Theorem~8.6.3]{BHS20}, where instead of the 
single layer potential \eqref{e-singlelayer_potential}
and the single layer boundary integral operator \eqref{e-singlelayer} the terminology of $\gamma$-fields, Weyl functions or $Q$-functions, and Dirichlet-to-Neumann maps is used.

\begin{lem} \label{lemma_Dirichlet_resolvent}
Let $\Omega$ and $\Omega_\pm$ be as in Hypothesis~\ref{hypothesis_Omega} and $-\Delta_D^\Omega$ and $-\Delta_D^{\Omega_\pm}$ be the corresponding Dirichlet Laplacians defined as in~\eqref{def_Dirichlet_op}. Then, for $-\Delta_D := (-\Delta_D^{\Omega_+}) \oplus (-\Delta_D^{\Omega_-})$ and any $z\in \rho(-\Delta_D)= \mathbb{C} \setminus  [0, \infty)$ the resolvent formula
  \begin{equation*}
    (-\Delta_D - z)^{-1}  = (-\Delta - z)^{-1} - SL_z \mathcal{S}_z^{-1} SL_{\overline{z}}^* 
  \end{equation*}
  holds. In particular, one has
  \begin{equation*}
    (-\Delta_D^\Omega - z)^{-1} = P_\Omega (-\Delta - z)^{-1} P_\Omega^* - P_\Omega SL_z \mathcal{S}_z^{-1} SL_{\overline{z}}^* P_\Omega^*
  \end{equation*}
  with the projection and embedding operators $P_\Omega$ and $P_\Omega^*$ from Corollary~\ref{corollary_H_tau_properties}.
\end{lem}

\section{The Nonrelativistic Limit}\label{s-nonrel}

In this section we compute the nonrelativistic limit of the operator $A_\kappa^\Sigma$ defined in~\eqref{def_delta_op} and use this to show Theorem~\ref{theorem_nr_limit} and Corollary~\ref{cor_Faber_Krahn}. Again, we will always assume that $\Omega_\pm$ is as in Hypothesis~\ref{hypothesis_Omega}  and $\Sigma = \partial \Omega_\pm$. Furthermore, we will often assume that $z \in \mathbb{C} \setminus [0, \infty)$ and $c>\sqrt{|z|}$, as then $z + \frac{c^2}{2} \in \rho(A_0)=\rho(A_\kappa^\Sigma)$; cf. 
Proposition~\ref{proposition_delta_op}.
In the following, the Krein type resolvent formula
\begin{equation} \label{krein_rescaled}
  \left(A_\kappa^\Sigma - \left(z + \frac{c^2}{2}\right) \right)^{-1}\! = \!\left(A_0 - \left(z + \frac{c^2}{2}\right) \right)^{-1} \! - \Phi_{z + c^2/2} \mathcal{M}_c \big( \vartheta_c + \mathcal{M}_c \mathcal{C}_{z + c^2/2} \mathcal{M}_c \big)^{-1} \mathcal{M}_c \Phi_{\overline{z} + c^2/2}^*
\end{equation}
from Proposition~\ref{proposition_delta_op} will play an important role.
We will compute the limit of each of the terms on the right hand side separately. The convergence of $(A_0 - (z + c^2/2))^{-1}$, $\Phi_{z+c^2/2} \mathcal{M}_c$, and $\mathcal{M}_c \Phi_{\overline{z} + c^2/2}^*$ is investigated in Section~\ref{section_convergence_A_0_Phi}, the convergence of  $( \vartheta_c + \mathcal{M}_c \mathcal{C}_{z + c^2/2} \mathcal{M}_c )^{-1}$ is treated in Section~\ref{section_convergence_C_inv}, and the nonrelativistic limit 
of $A_\kappa^\Sigma$ is computed in Section~\ref{abc}.

\subsection{Convergence of $(A_0 - (z + c^2/2))^{-1}$, $\Phi_{z+c^2/2} \mathcal{M}_c$, and $\mathcal{M}_c \Phi_{\overline{z} + c^2/2}^*$} \label{section_convergence_A_0_Phi}

First, the nonrelativistic limit of the free Dirac operator $A_0$ defined in~\eqref{e-freedirac} is discussed. This result is well-known,  it follows, e.g., as a special case of the results in \cite[Section~6]{T92}. However, since the result and the topology, in which the convergence takes place, are of importance in the analysis of $\Phi_{z + c^2/2} \mathcal{M}_c$ and $\mathcal{M}_c \Phi_{\overline{z} + c^2/2}^*$, we give a direct simple proof here to keep the presentation self-contained. 

\begin{prop}\label{p-freeconvergence}
Let $z \in \mathbb{C} \setminus [0, \infty)$ and $c>\sqrt{|z|}$. Then, there exists a constant $K(z) $ such that
\begin{align*}
   \bigg\| \left(A_0-\left(z+\frac{c^2}{2}\right)\right)^{-1}-(-\Delta-z)^{-1} \begin{pmatrix} I_2 & 0 \\ 0 & 0 \end{pmatrix} \bigg\|_{L^2(\RR^3; \mathbb{C}^4)\to H^1(\RR^3; \mathbb{C}^4)} \leq \frac{K(z)}{c}.
\end{align*}
\end{prop}

\begin{proof}
We shall use \eqref{norm_Sobolev_space} and compute for $f \in L^2(\RR^3; \mathbb{C}^4)$ and $c>\sqrt{|z|}$ 
\begin{align*}
&\int_{\RR^3}(1+|x|^2)\Bigg|\cF\left[\left(A_0-\left(z+\frac{c^2}{2}\right)\right)^{-1}-(-\Delta-z)^{-1}\begin{pmatrix} I_2 & 0 \\ 0 & 0 \end{pmatrix} \right] f (x)\Bigg|^2dx  \\
& \qquad =\int_{\RR^3}(1+|x|^2)\Bigg|\left[\frac{\al\cdot x+\frac{c}{2}\beta+\left(\frac{z}{c}+\frac{c}{2}\right)I_4}{c(|x|^2-(\frac{z^2}{c^2}+z))}-\frac{1}{|x|^2-z}\begin{pmatrix} I_2 & 0 \\ 0 & 0 \end{pmatrix} \right]\mathcal{F} f (x)\Bigg|^2 d x.
\end{align*}
Next, we decompose the part of the integrand that does not depend on $\mathcal{F} f$ in the last line of the equation. Using 
$\frac{1}{2}(\beta+I_4)=\big( \begin{smallmatrix} I_2 & 0 \\ 0 & 0 \end{smallmatrix} \big)$ we find
\begin{align*}
    \sup_{x\in\RR^3}&(1+|x|^2)\Bigg|\frac{\al\cdot x+\frac{z}{c}I_4}{c \big(|x|^2-\frac{z^2}{c^2}-z\big)}+\frac{\frac{1}{2}(\beta+I_4)}{|x|^2-\frac{z^2}{c^2}-z}-\frac{1}{|x|^2-z} \begin{pmatrix} I_2 & 0 \\ 0 & 0 \end{pmatrix}\Bigg|^2 \\
    &= \sup_{x\in\RR^3}\frac{1+|x|^2}{c^2}\Bigg|\frac{\al\cdot x+\frac{z}{c}I_4}{|x|^2-\frac{z^2}{c^2}-z}+\frac{z^2}{c (|x|^2-\frac{z^2}{c^2}-z) (|x|^2-z)} \begin{pmatrix} I_2 & 0 \\ 0 & 0 \end{pmatrix}\Bigg|^2\leq \frac{K(z)^2}{c^2}
\end{align*}
for some constant $K(z)$ that does not depend on $c$, since the assumptions $z \in \mathbb{C} \setminus [0, \infty)$ and $c>\sqrt{|z|}$ ensure that there is no singularity in the last $x$-dependent expression. As $\| \mathcal{F} f\|_{L^2(\RR^3; \mathbb{C}^4)} = \| f\|_{L^2(\RR^3; \mathbb{C}^4)}$ we conclude
\begin{align*}
\int_{\RR^3}\!(1+|x|^2)\Bigg|\cF \! \left[\left(A_0-\left(z+\frac{c^2}{2}\right)\right)^{-1} \! \! -(-\Delta-z)^{-1} \begin{pmatrix} I_2 & 0 \\ 0 & 0 \end{pmatrix} \right] \! f (x)\Bigg|^2 \! dx \leq \frac{K(z)^2}{c^2} \| f \|_{L^2(\mathbb{R}^3; \mathbb{C}^4)}^2,
\end{align*}
which  shows the desired result.
\end{proof}

By using the convergence result from Proposition~\ref{p-freeconvergence} and the definition~\eqref{def_Phi_z_star} of $\Phi_z^*$, it is not difficult to obtain the convergence of $\Phi_{z+c^2/2}$ and $\Phi^*_{z+c^2/2}$. Recall that $SL_\mu$, $\mu \in \mathbb{C} \setminus [0, \infty)$, is the single layer potential defined in~\eqref{e-singlelayer_potential} and that $\mathcal{M}_c$ is the scaling matrix given by~\eqref{def_M_c}. Since there is a multiplication by $\sqrt{c}$ involved, the rate of convergence in the following proposition reduces to $\mathcal{O}(c^{-1/2})$. This is the main reason why we get this rate of convergence in Theorem~\ref{theorem_nr_limit}.

\begin{prop}\label{p-allconvergence}
    Let $z \in \mathbb{C} \setminus [0, \infty)$ and $c>\sqrt{|z|}$. Then, there exists a constant $K(z)$ such that 
\begin{align}\label{e-phidifference}
    \left\|\Phi_{z+c^2/2} \mathcal{M}_c -SL_z \begin{pmatrix} I_2 & 0 \\ 0 & 0 \end{pmatrix} \right\|_{H^{-1/2}(\Sigma; \mathbb{C}^4)\to L^2(\RR^3; \mathbb{C}^4)}\leq \frac{K(z)}{\sqrt{c}}
\end{align}    
and
\begin{align}\label{e-phistardifference}
    \left\| \mathcal{M}_c \Phi^*_{z+c^2/2}- SL_z^*\begin{pmatrix} I_2 & 0 \\ 0 & 0 \end{pmatrix} \right\|_{L^2(\RR^3; \mathbb{C}^4)\to H^{1/2}(\Sigma; \mathbb{C}^4)} \leq \frac{K(z)}{\sqrt{c}}.
\end{align}
In particular, the operators $\Phi_{z+c^2/2} \mathcal{M}_c:H^{-1/2}(\Sigma; \mathbb{C}^4)\to L^2(\RR^3; \mathbb{C}^4)$ and the mappings
$\mathcal{M}_c \Phi^*_{z+c^2/2}:L^2(\RR^3; \mathbb{C}^4)\to H^{1/2}(\Sigma; \mathbb{C}^4)$ are uniformly bounded in $c$.
\end{prop}

\begin{proof}
  Recall from~\eqref{def_Phi_z_star} and~\eqref{mapping_properties_SL_mu_star} that
    \begin{align*}
    \Phi^*_{z+c^2/2}=\gamma_D\left(A_0-\left(\overline{z}+\frac{c^2}{2}\right)\right)^{-1}\quad \text{and}\quad
    SL_z^*=\gamma_D(-\Delta-\overline{z})^{-1}.
    \end{align*}
    Hence, \eqref{e-phistardifference} follows from Proposition~\ref{p-freeconvergence} and the mapping properties of the trace operator; the stated rates of convergence are obtained by accounting for the matrix terms in equation~\eqref{e-phistardifference}. The claim in~\eqref{e-phidifference} follows from~\eqref{e-phistardifference} by duality. The uniform boundedness of 
    $\Phi_{z+c^2/2} \mathcal{M}_c$ and $\mathcal{M}_c \Phi^*_{z+c^2/2}$ is clear as these operators converge.
\end{proof}

\subsection{Convergence of $( \vartheta_c + \mathcal{M}_c \mathcal{C}_{z + c^2/2} \mathcal{M}_c )^{-1}$} \label{section_convergence_C_inv}

The more difficult part in the analysis of~\eqref{krein_rescaled} is $( \vartheta_c + \mathcal{M}_c \mathcal{C}_{z + c^2/2} \mathcal{M}_c )^{-1}$. To handle it in the computation of the nonrelativistic limit, first a more detailed consideration of $\mathcal{C}_{z + c^2/2}$ is provided.
Define for $z \in \rho(A_0)$ the auxiliary operator $\cT_z$ that formally acts on a sufficiently smooth function $\varphi:\Sigma\to\CC^2$ via
\begin{align*}
    \cT_z\varphi(x):=\lim_{\eps\to 0^+}\int_{\Sigma\setminus B(x,\eps)}t_z(x-y)\varphi(y)d\sigma(y), \hspace{2em} x\in\Sigma,
\end{align*}
with
\begin{align*}
t_z(x):=\left(1-i\sqrt{\frac{z^2}{c^2}-\frac{c^2}{4}}|x|\right)\frac{i(\sigma\cdot x)}{4\pi |x|^3}e^{i\sqrt{z^2/c^2-c^2/4}|x|}, \quad x \neq 0.
\end{align*}
Next, the definition of $G_z$ in~\eqref{e-greens} implies 
\begin{align*}
    G_{z+c^2/2}(x)=\left(\frac{z}{c^2}I_4+\begin{pmatrix} I_2 & 0 \\ 0 & 0  \end{pmatrix} +\left(1-i\sqrt{z + \frac{z^2}{c^2}}|x|\right)\frac{i(\al\cdot x)}{c|x|^2}\right)\frac{e^{i\sqrt{z + z^2/c^2}|x|}}{4\pi|x|}.
\end{align*}
This and the definitions of $\cC_z$ and $\mathcal{S}_z$ in equations \eqref{e-cz} and \eqref{e-singlelayer} lead to 
\begin{align}\label{e-newcz}
\cC_{z+c^2/2}=\begin{pmatrix}
\big(\frac{z}{c^2}+1\big)\cS_{z+z^2/c^2} I_2 & \frac{1}{c}\cT_{z+c^2/2} \\
\frac{1}{c}\cT_{z+c^2/2} & \frac{z}{c^2}\cS_{z+z^2/c^2} I_2
\end{pmatrix}.
\end{align}
It follows from the latter representation and~\eqref{C_z_mapping_properties}--\eqref{C_z_dual} that $\cT_{z + c^2/2}$ gives rise to a bounded operator 
\begin{equation}\label{tttbbb}
\mathcal{T}_{z + c^2/2}: H^s(\Sigma; \mathbb{C}^2) \rightarrow H^s(\Sigma; \mathbb{C}^2),\quad s \in \bigl[-\tfrac{1}{2}, \tfrac{1}{2}\bigr],
\end{equation}
and that the anti-dual of $\mathcal{T}_{z + c^2/2}$ satisfies $\mathcal{T}_{z + c^2/2}' = \mathcal{T}_{\overline{z} + c^2/2}$. 
In the next proposition we show that these operators are even uniformly bounded in $c$.

\begin{prop}\label{p-tweakderiv}
  Let $z \in \mathbb{C} \setminus [0, \infty)$ and $c > \sqrt{|z|}$. Then, for any $s \in [-\frac{1}{2}, \frac{1}{2}]$ the operators $\cT_{z+c^2/2}: H^s(\Sigma; \mathbb{C}^2) \rightarrow H^s(\Sigma; \mathbb{C}^2)$ are uniformly bounded in $c$.
\end{prop}
\begin{proof}
The proof of this proposition is split into three steps. In \textit{Step~1} the integral kernel $t_{z+c^2/2}$ of $\mathcal{T}_{z + c^2/2}$ is decomposed into a singular part $d$, which is independent of $c$, and a remainder $\widetilde{t}_{z, c}$ which is easier to analyze. In \textit{Step~2} it is shown that the integral operator with kernel $\widetilde{t}_{z, c}$ gives rise to a bounded operator from $L^2(\Sigma; \mathbb{C}^2)$ to $H^1(\Omega_+; \mathbb{C}^2)$ that is uniformly bounded in $c$. By combining the results from \textit{Step~1 \& 2}, the proof of the proposition is completed in \textit{Step~3}.

\textit{Step~1.} 
Rewriting the exponential in the kernel 
\begin{align*}
t_{z + c^2/2}(x)=\left(1-i\sqrt{z+\frac{z^2}{c^2}}|x|\right)\frac{i(\sigma\cdot x)}{4\pi |x|^3}e^{i\sqrt{z+z^2/c^2}|x|}, \quad x \neq 0,
\end{align*}
as a power series shows that the terms with $|x|^{-2}$ cancel out. After combining and rearranging the coefficients of the remaining terms we obtain
\begin{equation}\label{e-power}
t_{z + c^2/2}(x)= d(x)+ \widetilde{t}_{z, c}(x),
\end{equation}
where 
\begin{align}\label{e-rkernel}
d(x)=\frac{i(\sigma\cdot x)}{4 \pi|x|^3},\quad x \neq 0, 
\end{align}
and
\begin{equation*} 
  \wt{t}_{z, c}(x)=\sum_{k=0}^\infty\left(\frac{i^{k+3}}{(k+2)!}+\frac{i^{k+1}}{(k+1)!}\right)\left(\sqrt{z + \frac{z^2}{c^2}}\right)^{k+2} |x|^{k-1} \frac{ \sigma\cdot x }{4\pi}, \quad x \neq 0.
\end{equation*}

\textit{Step~2}. 
Now we consider $\wt{t}_{z, c}(x-y)$ for $x\in\Omega_+$ and $y\in\Sigma$ and define the 
integral operator $\widetilde{T}_{z, c}$ for sufficiently smooth functions $\varphi: \Sigma \rightarrow \mathbb{C}^2$ as
\begin{equation} \label{def_tilde_T}
  \widetilde{T}_{z, c} \varphi(x) := \int_\Sigma \widetilde{t}_{z, c}(x-y) \varphi(y) d \sigma(y), \quad x \in \Omega_+.
\end{equation}
We will show that
\begin{equation} \label{tilde_T_t_mapping_properties}
  \widetilde{T}_{z, c}: L^2(\Sigma; \mathbb{C}^2) \rightarrow H^1(\Omega_+; \mathbb{C}^2) \quad \text{is uniformly bounded in } c.
\end{equation}
For this, we first establish some simple bounds on $\wt{t}_{z, c}$ and its first order derivatives that are independent of $c$. 
Observe that for a constant $K_1=K_1(z)$ one has for all $x \in \Omega_+$ and $y \in \Sigma$
\begin{equation} \label{e-tbound}
\begin{aligned}
\big|\wt{t}_{z, c}(x-y)\big|\leq \sum_{k=0}^\infty\frac{2}{k!}\frac{\big(\sqrt{2|z|}\big)^{k+2}|x - y|^k}{4\pi} \leq K_1,
\end{aligned}
\end{equation}
as the latter series is absolutely converging and defines a continuous function on the compact set $\overline{\Omega_+} \times \Sigma$.
Likewise, there exists a constant $K_2=K_2(z)$ such that for the partial derivatives of $\wt{t}_{z, c}$ and all $x \in \Omega_+$ and $y \in \Sigma$ one has
\begin{equation} \label{e-tderivbound}
    \begin{aligned}
        \big|\pt_{x_j} \wt{t}_{z, c}(x-y)\big|&\leq \sum_{k=0}^\infty \frac{2}{k!}\frac{\big(\sqrt{2|z|}\big)^{k+2}}{4\pi} \big|\pt_{x_j}\big((\sigma\cdot (x-y))|x-y|^{k-1}\big) \big| \\
        &= \sum_{k=0}^\infty \frac{\big(\sqrt{2|z|}\big)^{k+2}}{2\pi k!} \bigg|\sigma_j +\frac{(k-1)(x_j-y_j)(\sigma\cdot (x-y))}{|x-y|^2}\bigg||x-y|^{k-1} \\
        &\leq \frac{K_2}{|x-y|}.
    \end{aligned}
\end{equation}
Since $\Omega_+$ is bounded, \eqref{e-tbound} and~\eqref{e-tderivbound} imply $(x,y)\mapsto \wt{t}_{z, c}(x-y) \in L^2(\Omega_+ \times \Sigma; \mathbb{C}^{2 \times 2})$, $(x,y) \mapsto \pt_{x_j} \wt{t}_{z, c}(x-y) \in L^2(\Omega_+ \times \Sigma; \mathbb{C}^{2 \times 2})$ and there exists a constant $K_3=K_3(z)$ such that
\begin{equation} \label{L_2_bound_tilde_T_kernel}
  \int_{\Omega_+} \int_\Sigma \big|\widetilde{t}_{z, c}(x-y)\big|^2 d\sigma(y) dx \leq K_3, \qquad  \int_{\Omega_+} \int_\Sigma \big| \partial_{x_j} \widetilde{t}_{z, c}(x-y)\big|^2 d\sigma(y) dx \leq K_3.
\end{equation}
Furthermore, using that for any $y \in \Sigma$ one has $x \mapsto \widetilde{t}_{z, c}(x-y) \in C^\infty(\Omega_+; \mathbb{C}^{2 \times 2})$,~\eqref{e-tderivbound}, and the dominated convergence theorem, it is not difficult to see that for any $\varphi \in L^2(\Sigma; \mathbb{C}^2)$ one has $\widetilde{T}_{z, c} \varphi \in C^1(\Omega_+; \mathbb{C}^2)$
and
\begin{equation} \label{derivative_tilde_T}
  \partial_{x_j} \widetilde{T}_{z, c} \varphi(x) = \int_\Sigma \partial_{x_j} \widetilde{t}_{z, c}(x-y) \varphi(y) d \sigma(y), 
  \quad  x \in \Omega_+.
\end{equation}
Combining \eqref{def_tilde_T} and \eqref{derivative_tilde_T} with~\eqref{L_2_bound_tilde_T_kernel} shows that $\widetilde{T}_{z, c}, \partial_{x_j} \widetilde{T}_{z, c}: L^2(\Sigma; \mathbb{C}^2) \rightarrow L^2(\Omega_+; \mathbb{C}^2)$ are Hilbert-Schmidt operators that are uniformly bounded in $c$, and hence~\eqref{tilde_T_t_mapping_properties} is true.

\textit{Step~3}.
We verify that $\mathcal{T}_{z + c^2/2}: H^{s}(\Sigma; \mathbb{C}^2) \rightarrow H^{s}(\Sigma; \mathbb{C}^2)$ is uniformly bounded in $c$ for any 
$s \in [-\frac{1}{2}, \frac{1}{2}]$. First, we do this for $s = \frac{1}{2}$. For that purpose,
consider the operator $\gamma_D \widetilde{T}_{z, c}: L^2(\Sigma; \mathbb{C}^2) \rightarrow H^{1/2}(\Sigma; \mathbb{C}^2)$, which is uniformly bounded in 
$c$ by the results in \textit{Step~2}, and hence also the restriction 
\begin{equation}\label{uniform_bounded22}
\gamma_D \widetilde{T}_{z, c}: H^{1/2}(\Sigma; \mathbb{C}^2) \rightarrow H^{1/2}(\Sigma; \mathbb{C}^2)
\end{equation}
is uniformly bounded in 
$c$. Furthermore, we shall use that 
\begin{equation}\label{gamgam}
\gamma_D\widetilde{T}_{z, c}\varphi(x)=\int_\Sigma \widetilde{t}_{z, c}(x-y) \varphi(y) d \sigma(y), \quad x \in \Sigma,
\end{equation}
holds for all $\varphi\in L^2(\Sigma; \mathbb{C}^2)$ (and, in particular, for all $\varphi\in H^{1/2}(\Sigma; \mathbb{C}^2)$).
In fact, the estimate \eqref{e-tbound} extends to $\overline{\Omega_+}\times\Sigma$ and this implies that the function $\widetilde{T}_{z, c} \varphi:\Omega_+\rightarrow \mathbb{C}^2$ admits a continuous extension onto $\overline{\Omega_+}$, which shows \eqref{gamgam}.  

Next, recall that the function $d$ is defined by~\eqref{e-rkernel}. For $\varphi\in H^{1/2}(\Sigma; \mathbb{C}^2)$ consider the integral operator 
\begin{equation*}
  \mathcal{D} \varphi(x) := \lim_{\varepsilon \rightarrow 0^+} \int_{\Sigma \setminus B(x, \varepsilon)} d(x-y) \varphi(y) d \sigma(y), \quad x \in \Sigma,
\end{equation*}
which is bounded in $H^{1/2}(\Sigma; \mathbb{C}^2)$ by \cite[Theorem~4.3.1]{N01} as $d$ is a homogeneous kernel of order $0$ in the sense of \cite[Section~4.3.2]{N01}, see also \cite[Example~4.2]{N01} (the boundedness of $\mathcal D$ would also follow from 
\eqref{345} and the reasoning below, as 
$\mathcal{T}_{z + c^2/2}$ is bounded in $H^{1/2}(\Sigma; \mathbb{C}^2)$ by \eqref{tttbbb}). 
From~\eqref{e-power} and \eqref{gamgam} we obtain 
\begin{equation}\label{345}
\mathcal{T}_{z + c^2/2}\varphi = \mathcal{D}\varphi + \gamma_D \widetilde{T}_{z, c}\varphi,\quad \varphi\in H^{1/2}(\Sigma; \mathbb{C}^2),
\end{equation}
and now it follows from the uniform boundedness of the operator $\gamma_D \widetilde{T}_{z, c}$ in \eqref{uniform_bounded22} 
that also $\mathcal{T}_{z + c^2/2}: H^{1/2}(\Sigma; \mathbb{C}^2) \rightarrow H^{1/2}(\Sigma; \mathbb{C}^2)$ is uniformly bounded in $c$.

To show the claim for $s=-\frac{1}{2}$,  recall that $\mathcal{T}_{z + c^2/2}$ has a bounded extension in $H^{-1/2}(\Sigma; \mathbb{C}^2)$ given by $\mathcal{T}_{z + c^2/2} = (\mathcal{T}_{\overline{z} + c^2/2})'$. 
Hence, by the already shown uniform boundedness in $c$ of $\mathcal{T}_{\overline{z} + c^2/2}$ in $H^{1/2}(\Sigma; \mathbb{C}^2)$  also $(\mathcal{T}_{\overline{z} + c^2/2})' = \mathcal{T}_{z + c^2/2}: H^{-1/2}(\Sigma; \mathbb{C}^2) \rightarrow H^{-1/2}(\Sigma; \mathbb{C}^2)$ is uniformly bounded in $c$.
Finally, as $\mathcal{T}_{z + c^2/2}$ is uniformly bounded in $H^{-1/2}(\Sigma; \mathbb{C}^2)$ and $H^{1/2}(\Sigma; \mathbb{C}^2)$ in $c$, it follows with an interpolation argument using
\cite[Theorem B.2 and Theorem B.11]{M00} that $\mathcal{T}_{z + c^2/2}: H^{s}(\Sigma; \mathbb{C}^2) \rightarrow H^{s}(\Sigma; \mathbb{C}^2)$ is also 
uniformly bounded  in $c$ for any $s \in (-\frac{1}{2}, \frac{1}{2})$. This finishes the proof.
\end{proof}

Next, the convergence of $( \vartheta_c + \mathcal{M}_c \mathcal{C}_{z + c^2/2} \mathcal{M}_c )^{-1}$ is analyzed. Recall that $a_\pm$ is defined by~\eqref{def_coefficients}. By~\eqref{e-newcz} one has the block structure
\begin{align*}
\vartheta_c + \mathcal{M}_c \mathcal{C}_{z + c^2/2} \mathcal{M}_c =
\begin{pmatrix}
    \big( \frac{1}{c} a_+ +\left(\frac{z}{c^2}+1\right)\cS_{z+z^2/c^2} \big) I_2 & \frac{1}{\sqrt{c}}\cT_{z+c^2/2} \\
    \frac{1}{\sqrt{c}}\cT_{z+c^2/2} & \big(a_- +\frac{z}{c}\cS_{z+z^2/c^2}\big) I_2
\end{pmatrix}.
\end{align*}
To proceed, note that for $z \in \mathbb{C} \setminus [0, \infty)$ and $c > 0$ sufficiently large the operator 
$a_- +\frac{z}{c}\cS_{z+z^2/c^2}$ is boundedly invertible in $H^s(\Sigma; \mathbb{C})$, $s \in [-\frac{1}{2}, \frac{1}{2}]$, 
with inverse given by 
\begin{equation}\label{invi}
 \left(a_-+\frac{z}{c}\cS_{z+z^2/c^2}\right)^{-1}=\frac{1}{a_-}\sum_{n=0}^\infty\left(-\frac{z}{a_- c} \cS_{z+z^2/c^2}\right)^n,
\end{equation}
as $a_- < 0$ and by Lemma~\ref{lemma_single_layer_continuous} the operator $\cS_{z+z^2/c^2}$ is (uniformly) bounded in $H^s(\Sigma; \mathbb{C})$ in $c$.
Thus, one can write
\begin{equation}\label{decomposition_Schur_complement}
\begin{aligned}
\vartheta_c + \mathcal{M}_c \mathcal{C}_{z + c^2/2} \mathcal{M}_c =&
\begin{pmatrix}
I_2 & \frac{1}{\sqrt{c}}\cT_{z+c^2/2} \left(a_-+\frac{z}{c}\cS_{z+z^2/c^2}\right)^{-1} \\
0 & I_2
\end{pmatrix} \\
&\cdot
\begin{pmatrix}
    \wt{\cS}_{z, c} & 0 \\
    0 & (a_-+\frac{z}{c}\cS_{z+z^2/c^2}) I_2
\end{pmatrix}
\begin{pmatrix}
    I_2 & 0  \\
    \frac{1}{\sqrt{c}} \left(a_-+\frac{z}{c}\cS_{z+z^2/c^2}\right)^{-1} \cT_{z+c^2/2} & I_2
\end{pmatrix},
\end{aligned}
\end{equation}
where the Schur complement $\wt{\cS}_{z, c}$ is given by 
\begin{align}\label{e-schur}
\wt{\cS}_{z, c}=\frac{1}{c} a_+ I_2 + \left(\frac{z}{c^2}+1\right)\cS_{z+z^2/c^2} I_2 -\frac{1}{c}\cT_{z+c^2/2}\left(a_-+\frac{z}{c}\cS_{z+z^2/c^2}\right)^{-1}\cT_{z+c^2/2}.
\end{align}
The first and the third factor in~\eqref{decomposition_Schur_complement} are bijective in $H^{1/2}(\Sigma; \mathbb{C}^4)$. Since $\vartheta_c + \mathcal{M}_c \mathcal{C}_{z + c^2/2} \mathcal{M}_c$ has this property as well by Proposition~\ref{proposition_delta_op}, we conclude that also $\wt{\cS}_{z, c}$ is bijective in $H^{1/2}(\Sigma; \mathbb{C}^2)$. In the following proposition, the convergence of $\wt{\cS}^{-1}_{z, c}$ is analyzed.

\begin{prop}\label{proposition_sconverge}
Let $z < 0$ and $c>\sqrt{|z|}$.
Then, there exists a constant $K(z)$ such that for all $c$ sufficiently large 
\begin{align} \label{Schur_inverse_convergence}
    \big\|\wt{\cS}^{-1}_{z, c}-\cS_z^{-1}I_2\big\|_{H^{3/2}(\Sigma; \mathbb{C}^2)\to H^{-1/2}(\Sigma; \mathbb{C}^2)}\leq \frac{K(z)}{c}.
\end{align}
Moreover, $\wt{\cS}^{-1}_{z, c}:H^{1/2}(\Sigma; \mathbb{C}^2)\to H^{-1/2}(\Sigma; \mathbb{C}^2)$ is uniformly bounded in $c$. 
\end{prop}

\begin{proof}
The proof of this proposition is split into four steps. In \textit{Step~1} we show that for $s \in [-\frac{1}{2}, \frac{1}{2}]$ there exists a constant $K_1 = K_1(z, s)$ 
such that
\begin{equation} \label{convergence_Schur_complement}
  \big\| \wt{\cS}_{z, c}-\cS_z I_2 \big\|_{H^s(\Sigma; \mathbb{C}^2) \to H^s(\Sigma; \mathbb{C}^2)} \leq \frac{K_1}{c}
\end{equation}
for $c>0$ sufficiently large.
In \textit{Step~2} we verify that the realization of $\wt{\cS}_{z, c}$ in $L^2(\Sigma; \mathbb{C}^2)$ is bijective and there exists a constant $K_2 $ such that for $c>0$ sufficiently large 
\begin{equation} \label{bound_Schur_inverse_L2}
  \big\| \wt{\cS}_{z, c}^{-1} \big\|_{L^2(\Sigma; \mathbb{C}^2) \to L^2(\Sigma; \mathbb{C}^2)} \leq K_2 c.
\end{equation}
Using this, we show in \textit{Step~3} our claim that $\wt{\cS}^{-1}_{z, c}:H^{1/2}(\Sigma; \mathbb{C}^2)\to H^{-1/2}(\Sigma; \mathbb{C}^2)$ is uniformly bounded in $c$, while in \textit{Step~4} we prove \eqref{Schur_inverse_convergence}.

\textit{Step~1}. For $s \in [-\frac{1}{2}, \frac{1}{2}]$ fixed and $c>0$ sufficiently large we obtain the estimate
\begin{equation}\label{hoho}
  \begin{split}
    \big\| &\wt{\cS}_{z, c} -\cS_z I_2 \big\|_{H^s(\Sigma; \mathbb{C}^2) \to H^s(\Sigma; \mathbb{C}^2)} 
    \leq \left\| (\cS_{z+z^2/c^2} - \cS_z) I_2 \right\|_{H^s(\Sigma; \mathbb{C}^2) \to H^s(\Sigma; \mathbb{C}^2)} \\
    &+ \frac{1}{c}\left\|  a_+ I_2 + \frac{z}{c}\cS_{z+z^2/c^2} I_2 -\cT_{z+c^2/2}\left(a_-+\frac{z}{c}\cS_{z+z^2/c^2}\right)^{-1}\cT_{z+c^2/2} \right\|_{H^s(\Sigma; \mathbb{C}^2) \to H^s(\Sigma; \mathbb{C}^2)}
  \end{split}
\end{equation}
from \eqref{e-schur}. For the first term on the right-hand side of \eqref{hoho} one has by Lemma~\ref{lemma_single_layer_continuous} 
$$\bigl\| (\cS_{z+z^2/c^2} - \cS_z) I_2 \bigr\|_{H^s(\Sigma; \mathbb{C}^2) \to H^s(\Sigma; \mathbb{C}^2)}\leq 
\bigl\| (\cS_{z+z^2/c^2} - \cS_z) I_2 \bigr\|_{H^{-3/2}(\Sigma; \mathbb{C}^2) \to H^{3/2}(\Sigma; \mathbb{C}^2)}\leq K_1' \frac{z^2}{c^2}
$$
with some constant $K_1'= K_1'(z)$. Note also that $\cS_{z+z^2/c^2}$ is uniformly bounded in $H^s(\Sigma; \mathbb{C})$ for $c>0$ sufficiently large by Lemma~\ref{lemma_single_layer_continuous}.
Therefore,
since $a_- < 0$ we conclude from \eqref{invi} and the estimate
\begin{equation*}
 \left\|\left(a_-+\frac{z}{c}\cS_{z+z^2/c^2}\right)^{-1}\right\|_{H^s(\Sigma; \mathbb{C}) \to H^s(\Sigma; \mathbb{C})}
 \leq\frac{1}{-a_-} \left(1 - \frac{z}{a_-c}\Vert\cS_{z+z^2/c^2}\Vert_{H^s(\Sigma; \mathbb{C}) \to H^s(\Sigma; \mathbb{C})}\right)^{-1} 
 \end{equation*}
that $(a_-+\frac{z}{c}\cS_{z+z^2/c^2})^{-1}$ is also uniformly bounded in $H^s(\Sigma; \mathbb{C})$ for $c>0$ sufficiently large. 
Combining this with Proposition~\ref{p-tweakderiv} it follows that the second term on the right-hand side of \eqref{hoho} is bounded by $\frac{K_1''}{c}$ with some constant
$K_1'' = K_1''(z, s)$; thus we conclude \eqref{convergence_Schur_complement}.

\textit{Step~2}. For $z<0$ and $c>0$ sufficiently large we shall now consider the operator 
\begin{align*}
\wt{\cS}_{z, c}=\frac{1}{c} a_+ I_2 + \left(\frac{z}{c^2}+1\right)\cS_{z+z^2/c^2} I_2 -\frac{1}{c}\cT_{z+c^2/2}\left(a_-+\frac{z}{c}\cS_{z+z^2/c^2}\right)^{-1}\cT_{z+c^2/2}
\end{align*}
 in $L^2(\Sigma; \mathbb{C}^2)$. Note that for $c>0$ sufficiently large $\cS_{z+z^2/c^2}$ is bounded, self-adjoint and nonnegative in $L^2(\Sigma; \mathbb{C}^2)$ (see the discussion after \eqref{mapping_properties_S_mu}) and hence the same holds for the operator $(\frac{z}{c^2}+1)\cS_{z+z^2/c^2}$.
 Furthermore, for $c>0$ sufficiently large $\cT_{z+c^2/2}$
is bounded and self-adjoint in $L^2(\Sigma; \mathbb{C}^2)$ (see \eqref{tttbbb}), and together with \eqref{invi} we conclude that also  $\wt{\cS}_{z, c}$ is bounded and self-adjoint. 
As $a_-<0$ and $\cS_{z+z^2/c^2}$ is uniformly bounded in $c$ it is clear that $a_-+\frac{z}{c}\cS_{z+z^2/c^2}$ 
is a negative operator in $L^2(\Sigma; \mathbb{C}^2)$ for $c>0$ sufficiently large, and the same is true for its inverse. Therefore, 
$$
-\frac{1}{c}\cT_{z+c^2/2}\left(a_-+\frac{z}{c}\cS_{z+z^2/c^2}\right)^{-1}\cT_{z+c^2/2}
$$
is a nonnegative operator in $L^2(\Sigma; \mathbb{C}^2)$ for $c>0$ sufficiently large. This implies $\wt{\cS}_{z, c}\geq\frac{a_+}{c}$ for $c>0$ sufficiently large,
which in turn yields \eqref{bound_Schur_inverse_L2} with $K_2=a_+^{-1}$.

\textit{Step~3}. We claim that $\wt{\cS}^{-1}_{z, c}: H^{1/2}(\Sigma; \mathbb{C}^2) \to H^{-1/2}(\Sigma; \mathbb{C}^2)$ is uniformly bounded in $c$. 
For this it suffices to prove that for $c>0$ sufficiently large there exists a constant $K_3=K_3(z)$ such that
\begin{align}\label{e-interp1}
    \big\| \wt{\cS}_{z, c}^{-1} \big\|_{H^{1}(\Sigma; \mathbb{C}^2) \to L^2(\Sigma; \mathbb{C}^2)} \leq K_3,
\end{align}
as then by duality and formal symmetry one also has 
\begin{align*}
    \big\| \wt{\cS}_{z, c}^{-1} \big\|_{L^2(\Sigma; \mathbb{C}^2) \to H^{-1}(\Sigma; \mathbb{C}^2)} \leq K_3,
\end{align*}
and an interpolation argument (see \cite[Theorem B.2 and Theorem B.11]{M00}) leads to the assertion.

To show~\eqref{e-interp1}, we use
\begin{align}\label{e-mapping1}
    \wt{\cS}_{z, c}^{-1}=\cS_z^{-1} I_2 -\wt{\cS}_{z, c}^{-1}\left(\wt{\cS}_{z, c} - \cS_z I_2\right) \cS_z^{-1} I_2
\end{align}
and the fact that $\cS_z^{-1}:H^{1}(\Sigma; \mathbb{C})\rightarrow L^2(\Sigma; \mathbb{C})$ is bounded; cf. \eqref{mapping_properties_S_mu}.
Using~\eqref{convergence_Schur_complement} for $s=0$ with $K_1=K_1(z,0)$ and~\eqref{bound_Schur_inverse_L2} we obtain
\begin{equation*}
  \begin{split}
     \big\|  \wt{\cS}_{z, c}^{-1}& \big\|_{H^{1}(\Sigma; \mathbb{C}^2) \to L^2(\Sigma; \mathbb{C}^2)} \leq  \|\cS^{-1}_z\|_{H^1(\Sigma; \mathbb{C})\to L^2(\Sigma; \mathbb{C})}  \\
     &\quad + \big\| \wt{\cS}^{-1}_{z, c} \big\|_{L^2(\Sigma; \mathbb{C}^2) \to L^2(\Sigma; \mathbb{C}^2)} \big\| \cS_z I_2 - \wt{\cS}_{z, c} \big\|_{L^2(\Sigma; \mathbb{C}^2) \to L^2(\Sigma; \mathbb{C}^2)} \| \cS^{-1}_z\|_{H^{1}(\Sigma; \mathbb{C}) \to L^2(\Sigma; \mathbb{C})} \\
     &\leq \|\cS^{-1}_z\|_{H^1(\Sigma; \mathbb{C})\to L^2(\Sigma; \mathbb{C})}\left(1+K_2 \cdot c \cdot \frac{K_1}{c} \right),
  \end{split}
\end{equation*}
and hence \eqref{e-interp1} holds.

\textit{Step~4}. Finally, we show~\eqref{Schur_inverse_convergence}. Using again~\eqref{e-mapping1}, the fact that $\mathcal{S}_z: H^{1/2}(\Sigma; \mathbb{C}) \rightarrow H^{3/2}(\Sigma; \mathbb{C})$ is boundedly invertible, and the results from \textit{Step~1} and \textit{Step~3} we obtain
\begin{align*}
    &\big\| \wt{\cS}^{-1}_{z, c}-\cS^{-1}_z I_2 \big\|_{H^{3/2}(\Sigma; \mathbb{C}^2) \to H^{-1/2}(\Sigma; \mathbb{C}^2)} \\
    &\quad \leq \big\| \wt{\cS}^{-1}_{z, c} \big\|_{H^{1/2}(\Sigma; \mathbb{C}^2) \to H^{-1/2}(\Sigma; \mathbb{C}^2)} \big\| \cS_z I_2 - \wt{\cS}_{z, c} \big\|_{H^{1/2}(\Sigma; \mathbb{C}^2) \to H^{1/2}(\Sigma; \mathbb{C}^2)} \| \cS^{-1}_z\|_{H^{3/2}(\Sigma; \mathbb{C}) \to H^{1/2}(\Sigma; \mathbb{C})}\\
    &\quad \leq \frac{K(z)}{c}.
\end{align*}
This completes the proof of Proposition~\ref{proposition_sconverge}.
\end{proof}

Now we are ready to study the convergence of the inverse of $\vartheta_c + \mathcal{M}_c \mathcal{C}_{z + c^2/2} \mathcal{M}_c$.

\begin{prop} \label{proposition_convergence_inverse}
Let $z < 0$ and $c>\sqrt{|z|}$.
Then, there exists a constant $K(z)$ such that for all $c$ sufficiently large 
  \begin{equation*}
    \left\| \big( \vartheta_c + \mathcal{M}_c \mathcal{C}_{z + c^2/2} \mathcal{M}_c \big)^{-1} - \begin{pmatrix} \mathcal{S}_z^{-1} I_2 & 0 \\ 0 & a_-^{-1} I_2 \end{pmatrix} \right\|_{H^{3/2}(\Sigma; \mathbb{C}^4) \rightarrow H^{-1/2}(\Sigma; \mathbb{C}^4)} \leq \frac{K(z)}{\sqrt{c}}.
  \end{equation*}
  Moreover, $\big( \vartheta_c + \mathcal{M}_c \mathcal{C}_{z + c^2/2} \mathcal{M}_c \big)^{-1}: H^{1/2}(\Sigma; \mathbb{C}^4) \to H^{-1/2}(\Sigma; \mathbb{C}^4)$ is uniformly bounded in 
  $c$.
\end{prop}
\begin{proof}
  It follows from~\eqref{decomposition_Schur_complement} that
  \begin{equation}\label{e-inverted}
\begin{aligned}
\big(\vartheta_c + \mathcal{M}_c \mathcal{C}_{z + c^2/2} \mathcal{M}_c \big)^{-1}&= F_1(c) F_2(c) F_3(c),
\end{aligned}
\end{equation}
where
\begin{equation*}
\begin{split} 
F_1(c) &:= \begin{pmatrix}
I_2 & 0 \\
-\frac{1}{\sqrt{c}}\left(a_-+\frac{z}{c}\cS_{z+c^2/2}\right)^{-1}\cT_{z+c^2/2} & I_2
\end{pmatrix}, \\
F_2(c) &:=
\begin{pmatrix}
    \wt{\cS}_{z, c}^{-1} & 0 \\
    0 & \left(a_-+\frac{z}{c}\cS_{z+z^2/c^2}\right)^{-1} I_2
\end{pmatrix}, \\
F_3(c) &:= 
\begin{pmatrix}
    I_2 & -\frac{1}{\sqrt{c}} \cT_{z+c^2/2} \left(a_-+\frac{z}{c}\cS_{z+z^2/c^2}\right)^{-1}  \\
    0 & I_2
\end{pmatrix}.
\end{split}
\end{equation*}
  For $c>0$ sufficiently large we use the uniform boundedness of $\cS_{z+z^2/c^2}$ in $H^s(\Sigma; \mathbb{C})$, $s \in [-\frac{1}{2}, \frac{1}{2}]$, from 
  Lemma~\ref{lemma_single_layer_continuous} to estimate
  \begin{equation*} 
    \begin{split}
    &\left\| \left(a_-+\frac{z}{c}\cS_{z+z^2/c^2}\right)^{-1} - a_-^{-1} \right\|_{H^s(\Sigma; \mathbb{C}) \to H^s(\Sigma; \mathbb{C})} \\
    &\qquad\qquad\qquad\qquad =  -\frac{1}{a_-}\left\|\sum_{n=1}^{\infty} \left(-\frac{z}{a_- c} \cS_{z+z^2/c^2}\right)^n \right\|_{H^s(\Sigma; \mathbb{C}) \to H^s(\Sigma; \mathbb{C})}\\
    &\qquad\qquad\qquad\qquad\leq -\frac{1}{a_-}\left(\frac{\frac{z}{a_-c} \Vert\cS_{z+z^2/c^2} \Vert_{H^s(\Sigma; \mathbb{C}) \to H^s(\Sigma; \mathbb{C})}}{1-\frac{z}{a_- c}\Vert\cS_{z+z^2/c^2} \Vert_{H^s(\Sigma; \mathbb{C}) \to H^s(\Sigma; \mathbb{C})}}\right) \\ 
    &\qquad\qquad\qquad\qquad\leq \frac{K_1}{c},
    \end{split}
  \end{equation*}
  where $K_1=K_1(z,s)$ is a constant; for the restriction onto $H^{3/2}(\Sigma; \mathbb{C})$ viewed as a mapping into $H^{-1/2}(\Sigma; \mathbb{C})$ this estimate yields
  \begin{equation}\label{oderso}
   \left\| \left(a_-+\frac{z}{c}\cS_{z+z^2/c^2}\right)^{-1} - a_-^{-1} \right\|_{H^{3/2}(\Sigma; \mathbb{C}) \to H^{-1/2}(\Sigma; \mathbb{C})}\leq\frac{K_1}{c},
  \end{equation}
and also shows that $(a_-+\frac{z}{c}\cS_{z+z^2/c^2})^{-1}$ is 
  uniformly bounded in $H^s(\Sigma; \mathbb{C}^2)$, $s \in [-\frac{1}{2}, \frac{1}{2}]$, for $c>0$ sufficiently large; cf. \textit{Step~1} in the proof of Proposition~\ref{proposition_sconverge}. Together with Proposition~\ref{p-tweakderiv} this implies with a constant $K_2=K_2(z)$ that
  \begin{equation} \label{estimate_F_1_3}
    \| F_1(c) - I_4 \|_{H^{-1/2}(\Sigma; \mathbb{C}^4) \to H^{-1/2}(\Sigma; \mathbb{C}^4)} \leq \frac{K_2}{\sqrt{c}} \quad \text{and} \quad 
    \| F_3(c) - I_4 \|_{H^{1/2}(\Sigma; \mathbb{C}^4) \to H^{1/2}(\Sigma; \mathbb{C}^4)} \leq \frac{K_2}{\sqrt{c}}.
  \end{equation}
  In particular, $F_1(c)$ is uniformly bounded in $H^{-1/2}(\Sigma; \mathbb{C}^4)$ and $F_3(c)$ is uniformly bounded in $H^{1/2}(\Sigma; \mathbb{C}^4)$ in $c$,
  and the restrictions onto $H^{1/2}(\Sigma; \mathbb{C}^4)$ and $H^{3/2}(\Sigma; \mathbb{C}^4)$ satisfy the same bounds 
  \begin{equation} \label{estimate_F_1_3++}
    \| F_1(c) - I_4 \|_{H^{1/2}(\Sigma; \mathbb{C}^4) \to H^{-1/2}(\Sigma; \mathbb{C}^4)} \leq \frac{K_2}{\sqrt{c}} \quad \text{and} \quad 
    \| F_3(c) - I_4 \|_{H^{3/2}(\Sigma; \mathbb{C}^4) \to H^{1/2}(\Sigma; \mathbb{C}^4)} \leq \frac{K_2}{\sqrt{c}}.
  \end{equation}
  Moreover, Proposition~\ref{proposition_sconverge} and \eqref{oderso} imply
  \begin{equation*}
    \left\| F_2(c) - \begin{pmatrix} \mathcal{S}_z^{-1} I_2 & 0 \\ 0 & a_-^{-1} I_2 \end{pmatrix} \right\|_{H^{3/2}(\Sigma; \mathbb{C}^4) \rightarrow H^{-1/2}(\Sigma; \mathbb{C}^4)} \leq \frac{K_3}{c}
  \end{equation*}
  with some constant $K_3=K_3(z)$.
  Eventually, it follows from Proposition~\ref{proposition_sconverge} and the uniform boundedness of $(a_-+\frac{z}{c}\cS_{z+z^2/c^2})^{-1}$ as a mapping from
  $H^{1/2}(\Sigma; \mathbb{C})$ to $H^{-1/2}(\Sigma; \mathbb{C})$
  that $F_2(c): H^{1/2}(\Sigma; \mathbb{C}^4) \rightarrow H^{-1/2}(\Sigma; \mathbb{C}^4)$ is uniformly bounded.
  Combining this with~\eqref{estimate_F_1_3} and \eqref{estimate_F_1_3++} gives
  \begin{equation*}
    \begin{split}
      &\left\| \big( \vartheta_c + \mathcal{M}_c \mathcal{C}_{z + c^2/2} \mathcal{M}_c \big)^{-1} - \begin{pmatrix} \mathcal{S}_z^{-1} & 0 \\ 0 & a_-^{-1} I_2 \end{pmatrix} \right\|_{H^{3/2}(\Sigma; \mathbb{C}^4) \rightarrow H^{-1/2}(\Sigma; \mathbb{C}^4)} \\
      & \qquad\leq  \big\| F_1(c) F_2(c) (F_3(c) - I_4) \big\|_{H^{3/2}(\Sigma; \mathbb{C}^4) \rightarrow H^{-1/2}(\Sigma; \mathbb{C}^4)} \\
      &\qquad \quad + \left\| F_1(c) \left( F_2(c) - \begin{pmatrix} \mathcal{S}_z^{-1} I_2 & 0 \\ 0 & a_-^{-1} I_2 \end{pmatrix} \right) \right\|_{H^{3/2}(\Sigma; \mathbb{C}^4) \rightarrow H^{-1/2}(\Sigma; \mathbb{C}^4)} \\
      &\qquad \quad + \left\| (F_1(c) - I_4) \begin{pmatrix} \mathcal{S}_z^{-1} I_2 & 0 \\ 0 & a_-^{-1} I_2 \end{pmatrix} \right\|_{H^{3/2}(\Sigma; \mathbb{C}^4) \rightarrow H^{-1/2}(\Sigma; \mathbb{C}^4)} \\
      & \qquad\leq  \big\| F_1(c) F_2(c) \big\|_{H^{1/2}(\Sigma; \mathbb{C}^4) \rightarrow H^{-1/2}(\Sigma; \mathbb{C}^4)} \big\| F_3(c) - I_4 \big\|_{H^{3/2}(\Sigma; \mathbb{C}^4) \rightarrow H^{1/2}(\Sigma; \mathbb{C}^4)} \\
      &\qquad \quad + \left\| F_1(c) \right\|_{H^{-1/2}(\Sigma; \mathbb{C}^4) \rightarrow H^{-1/2}(\Sigma; \mathbb{C}^4)} \left\|  F_2(c) - \begin{pmatrix} \mathcal{S}_z^{-1} I_2 & 0 \\ 0 & a_-^{-1} I_2 \end{pmatrix}  \right\|_{H^{3/2}(\Sigma; \mathbb{C}^4) \rightarrow H^{-1/2}(\Sigma; \mathbb{C}^4)} \\
      &\qquad \quad + \left\| F_1(c) - I_4 \right\|_{H^{1/2}(\Sigma; \mathbb{C}^4) \rightarrow H^{-1/2}(\Sigma; \mathbb{C}^4)} \left\| \begin{pmatrix} \mathcal{S}_z^{-1} I_2 & 0 \\ 0 & a_-^{-1} I_2 \end{pmatrix} \right\|_{H^{3/2}(\Sigma; \mathbb{C}^4) \rightarrow H^{1/2}(\Sigma; \mathbb{C}^4)} \\
      &\qquad \leq \frac{K(z)}{\sqrt{c}},
    \end{split}
  \end{equation*}
  which is exactly the claimed convergence result.
  
  Finally, the claim about the uniform boundedness of the operator $( \vartheta_c + \mathcal{M}_c \mathcal{C}_{z + c^2/2} \mathcal{M}_c )^{-1}: H^{1/2}(\Sigma; \mathbb{C}^4) \to H^{-1/2}(\Sigma; \mathbb{C}^4)$ follows from~\eqref{e-inverted} and the above observations on the uniform boundedness of $F_1(c)$ in $H^{-1/2}(\Sigma; \mathbb{C}^4)$, $F_2(c)$ from 
  $H^{1/2}(\Sigma; \mathbb{C}^4)$ to $H^{-1/2}(\Sigma; \mathbb{C}^4)$, and 
  $F_3(c)$ in $H^{1/2}(\Sigma; \mathbb{C}^4)$.
\end{proof}

\subsection{Nonrelativistic limit of $A_\kappa^\Sigma$}\label{abc}

With the preparations from the previous sections we are now ready to 
discuss the nonrelativistic limit of the Dirac operators $A_\kappa^\Sigma$. 

\begin{prop}\label{proposition_confinement_nr_limit}
Let $A_\kappa^\Sigma$, $\kappa\in\mathbb R$, be as in \eqref{def_delta_op} and $-\Delta_D := (-\Delta_D^{\Omega_+}) \oplus (-\Delta_D^{\Omega_-})$, where $-\Delta_D^{\Omega_\pm}$ is the Dirichlet Laplacian in $\Omega_\pm$ from \eqref{def_Dirichlet_op}.
Let $z <0$ and $c > \sqrt{|z|}$. Then, there exists a constant $K(z)$ such that 
for all $c$ sufficiently large
\begin{align*}
\left\|\left(A_\kappa^\Sigma-\left(z+\frac{c^2}{2}\right)\right)^{-1}-(-\Delta_D-z)^{-1}\begin{pmatrix} I_2 & 0 \\ 0 & 0 \end{pmatrix} \right\|_{L^2(\RR^3; \mathbb{C}^4)\to L^2(\RR^3; \mathbb{C}^4)} \leq \frac{K(z)}{\sqrt{c}}.
\end{align*}
\end{prop}
\begin{proof}
  Let $\mathcal{M}_c$ and $\vartheta_c$ be defined by~\eqref{def_M_c}. As $-c^2<z<0$, one has $z \in \rho(-\Delta_D)$ and $z + \frac{c^2}{2} \in (-\frac{c^2}{2}, \frac{c^2}{2}) \subset \rho(A_\kappa^\Sigma)$; cf Proposition~\ref{proposition_delta_op}. Furthermore, from Proposition~\ref{proposition_delta_op} and Lemma~\ref{lemma_Dirichlet_resolvent}  we obtain
  \begin{equation} \label{resolvent_difference}
    \begin{split}
      &\left(A_\kappa^\Sigma-\left(z+\frac{c^2}{2}\right)\right)^{-1}-(-\Delta_D-z)^{-1}\begin{pmatrix} I_2 & 0 \\ 0 & 0 \end{pmatrix} \\
      &\qquad= \left(A_0 -\left(z+\frac{c^2}{2}\right)\right)^{-1} - \Phi_{z + c^2/2} \mathcal{M}_c \big( \vartheta_c + \mathcal{M}_c \mathcal{C}_{z + c^2/2} \mathcal{M}_c \big)^{-1} \mathcal{M}_c \Phi_{z + c^2/2}^* \\
      &\qquad \qquad - \left( (-\Delta - z)^{-1} -  SL_z \mathcal{S}_z^{-1} SL_{z}^* \right) \begin{pmatrix} I_2 & 0 \\ 0 & 0 \end{pmatrix} \\
      &\qquad = D_1(c) + D_2(c) + D_3(c) + D_4(c)
    \end{split}
  \end{equation}
  with
  \begin{equation*}
    \begin{split}
      D_1(c) &:= \left(A_0 -\left(z+\frac{c^2}{2}\right)\right)^{-1} - (-\Delta - z)^{-1} \begin{pmatrix} I_2 & 0 \\ 0 & 0 \end{pmatrix}, \\
      D_2(c) &:= - \Phi_{z + c^2/2} \mathcal{M}_c \big( \vartheta_c + \mathcal{M}_c \mathcal{C}_{z + c^2/2} \mathcal{M}_c \big)^{-1} \left( \mathcal{M}_c \Phi_{z + c^2/2}^* - SL_{z}^* \begin{pmatrix} I_2 & 0 \\ 0 & 0 \end{pmatrix} \right), \\
      D_3(c) &:= - \Phi_{z + c^2/2} \mathcal{M}_c \left( \big( \vartheta_c + \mathcal{M}_c \mathcal{C}_{z + c^2/2} \mathcal{M}_c \big)^{-1} - \begin{pmatrix} \mathcal{S}_z^{-1} & 0 \\ 0 & a_-^{-1} I_2 \end{pmatrix} \right)  SL_{z}^* \begin{pmatrix} I_2 & 0 \\ 0 & 0 \end{pmatrix}, \\
      D_4(c) &:= - \left( \Phi_{z + c^2/2} \mathcal{M}_c - SL_{z} \begin{pmatrix} I_2 & 0 \\ 0 & 0 \end{pmatrix} \right) \begin{pmatrix} \mathcal{S}_z^{-1} & 0 \\ 0 & a_-^{-1} I_2 \end{pmatrix}  SL_{z}^* \begin{pmatrix} I_2 & 0 \\ 0 & 0 \end{pmatrix}.
    \end{split}
  \end{equation*}
  First, it follows from Proposition~\ref{p-freeconvergence} that $\| D_1(c) \|_{L^2(\mathbb{R}^3; \mathbb{C}^4) \rightarrow L^2(\mathbb{R}^3; \mathbb{C}^4)} \leq \frac{K_1}{c}$ for a constant $K_1=K_1(z)$.
  To discuss $D_2(c)$ recall that $\Phi_{z + c^2/2} \mathcal{M}_c: H^{-1/2}(\Sigma; \mathbb{C}^4) \rightarrow L^2(\mathbb{R}^3; \mathbb{C}^4)$ is uniformly bounded in $c$ by Proposition~\ref{p-allconvergence} and 
  $( \vartheta_c + \mathcal{M}_c \mathcal{C}_{z + c^2/2} \mathcal{M}_c )^{-1}: H^{1/2}(\Sigma; \mathbb{C}^4) \rightarrow H^{-1/2}(\Sigma; \mathbb{C}^4)$ is uniformly bounded in $c$ by Proposition~\ref{proposition_convergence_inverse}. Hence, we find with Proposition~\ref{p-allconvergence} that there exists a constant $K_2 =K_2(z)$ such that
  \begin{equation*} 
    \begin{split}
      \| D_2(c) \|_{L^2(\mathbb{R}^3; \mathbb{C}^4) \rightarrow L^2(\mathbb{R}^3; \mathbb{C}^4)} &\leq \big\| \Phi_{z + c^2/2} \mathcal{M}_c \big\|_{H^{-1/2}(\Sigma; \mathbb{C}^4) \rightarrow L^2(\mathbb{R}^3; \mathbb{C}^4)} \\
      & \cdot \Big\| \big( \vartheta_c + \mathcal{M}_c \mathcal{C}_{z + c^2/2} \mathcal{M}_c \big)^{-1} \Big\|_{H^{1/2}(\Sigma; \mathbb{C}^4) \rightarrow H^{-1/2}(\Sigma; \mathbb{C}^4)} \\
      & \cdot\left\| \mathcal{M}_c \Phi_{z + c^2/2}^* - SL_{z}^* \begin{pmatrix} I_2 & 0 \\ 0 & 0 \end{pmatrix} \right\|_{L^2(\mathbb{R}^3; \mathbb{C}^4) \rightarrow H^{1/2}(\Sigma; \mathbb{C}^4)} \leq \frac{K_2}{\sqrt{c}}.
    \end{split}
  \end{equation*}
  Next, as $SL_{z}^*: L^2(\mathbb{R}^3; \mathbb{C}) \rightarrow H^{3/2}(\Sigma; \mathbb{C})$ is bounded (see \eqref{mapping_properties_SL_mu_star}), Proposition~\ref{proposition_convergence_inverse} implies that there exists a constant 
  $K_3 =K_3(z)$ such that
  \begin{equation*}
    \begin{split}
      \| D_3(c) &\|_{L^2(\mathbb{R}^3; \mathbb{C}^4) \rightarrow L^2(\mathbb{R}^3; \mathbb{C}^4)} \leq \big\| \Phi_{z + c^2/2} \mathcal{M}_c \big\|_{H^{-1/2}(\Sigma; \mathbb{C}^4) \rightarrow L^2(\mathbb{R}^3; \mathbb{C}^4)} \\
      &\qquad \qquad \cdot \left\| \big( \vartheta_c + \mathcal{M}_c \mathcal{C}_{z + c^2/2} \mathcal{M}_c \big)^{-1} - \begin{pmatrix} \mathcal{S}_z^{-1} & 0 \\ 0 & a_-^{-1} I_2 \end{pmatrix} \right\|_{H^{3/2}(\Sigma; \mathbb{C}^4) \rightarrow H^{-1/2}(\Sigma; \mathbb{C}^4)} \\
      &\qquad \qquad \cdot\left\| SL_{z}^* \begin{pmatrix} I_2 & 0 \\ 0 & 0 \end{pmatrix} \right\|_{L^2(\mathbb{R}^3; \mathbb{C}^4) \rightarrow H^{3/2}(\Sigma; \mathbb{C}^4)} \leq \frac{K_3}{\sqrt{c}}.
    \end{split}
  \end{equation*}
  In a similar way, as $\mathcal{S}_z^{-1}: H^{1/2}(\Sigma; \mathbb{C}) \rightarrow H^{-1/2}(\Sigma; \mathbb{C})$ is bounded (see~\eqref{mapping_properties_S_mu}), we find with Proposition~\ref{p-allconvergence} that there exists a constant $K_4 =K_4(z)$ such that
  \begin{equation*} 
    \begin{split}
      \| D_4(c) \|_{L^2(\mathbb{R}^3; \mathbb{C}^4) \rightarrow L^2(\mathbb{R}^3; \mathbb{C}^4)} \leq &\left\| \Phi_{z + c^2/2} \mathcal{M}_c - SL_{z} \begin{pmatrix} I_2 & 0 \\ 0 & 0 \end{pmatrix} \right\|_{H^{-1/2}(\Sigma; \mathbb{C}^4) \rightarrow L^2(\mathbb{R}^3; \mathbb{C}^4)} \\
      & \cdot \left\| \begin{pmatrix} \mathcal{S}_z^{-1} & 0 \\ 0 & a_-^{-1} I_2 \end{pmatrix} \right\|_{H^{1/2}(\Sigma; \mathbb{C}^4) \rightarrow H^{-1/2}(\Sigma; \mathbb{C}^4)} \\
      & \cdot\left\| SL_{z}^* \begin{pmatrix} I_2 & 0 \\ 0 & 0 \end{pmatrix} \right\|_{L^2(\mathbb{R}^3; \mathbb{C}^4) \rightarrow H^{1/2}(\Sigma; \mathbb{C}^4)} \leq \frac{K_4}{\sqrt{c}}.
    \end{split}
  \end{equation*}
  Now the statement of the proposition follows by combining
 the above estimates for the operators  $D_1(c), D_2(c), D_3(c)$, and $D_4(c)$ with~\eqref{resolvent_difference}.
\end{proof}

\subsection{Proof of Theorem~\ref{theorem_nr_limit} and Corollary~\ref{cor_Faber_Krahn}} \label{section_proofs_theorems}

In this section we complete the proof of our main result by combining
Lemma~\ref{lemma_confinement} and Proposition~\ref{proposition_confinement_nr_limit}. 
For this let $\kappa\in\mathbb R$, $\Omega, \Omega_\pm$ be as in Hypothesis~\ref{hypothesis_Omega}, and $\Sigma = \partial \Omega$. Let the Dirac operator $A_\kappa^{\Sigma}$ be defined as in \eqref{def_delta_op} and denote by $P_{\Omega}: L^2(\mathbb{R}^3; \mathbb{C}^4) \rightarrow L^2(\Omega; \mathbb{C}^4)$ the operator  $P_{\Omega} f = f \upharpoonright \Omega$; cf. Corollary~\ref{corollary_H_tau_properties}. Then, the self-adjoint Dirac operator $H_\kappa^{\Omega}$ in $L^2(\Omega; \mathbb{C}^4)$ satisfies  $$H_\kappa^{\Omega} = P_{\Omega} A_\kappa^{\Sigma} P_{\Omega}^*$$ by Lemma~\ref{lemma_confinement} and 
since $\sigma(H_\kappa^{\Omega})\subset (-\infty,-\tfrac{c^2}{2}]\cup[\tfrac{c^2}{2},\infty)$ by Corollary~\ref{corollary_H_tau_properties} 
it is clear that $z+c^2/2$ with $z\in\mathbb C\setminus[0,\infty)$ and $c>\sqrt{\vert z\vert}$ belongs to 
$\rho(H_\kappa^{\Omega})\cap\rho( A_\kappa^{\Sigma})$, so that
$$\left(H_\kappa^{\Omega} - \left(z + \frac{c^2}{2}\right) \right)^{-1} = P_{\Omega} \left(A_\kappa^{\Sigma} - \left(z + \frac{c^2}{2}\right) \right)^{-1} P_{\Omega}^*.$$
Therefore, if $z <0$, then \eqref{equation_nr_limit} follows from Proposition~\ref{proposition_confinement_nr_limit}.
In the general case $z\in\mathbb C\setminus[0,\infty)$ we obtain \eqref{equation_nr_limit} by assuming that $c > 1$ and using the identity
  \begin{equation*}
    \begin{split}
      &\left(H_\kappa^{\Omega} - \left(z + \frac{c^2}{2} \right) \right)^{-1} - (-\Delta_D^{\Omega} - z)^{-1} \begin{pmatrix} I_2 & 0 \\ 0 & 0 \end{pmatrix} = \left( I_4 + (z+1) (-\Delta_D^{\Omega} - z)^{-1} \begin{pmatrix} I_2 & 0 \\ 0 & 0 \end{pmatrix} \right) \\
      &~ \cdot \left[ \left(H_\kappa^{\Omega} - \! \left(-1 + \frac{c^2}{2}\right)  \right)^{-1}\! - (-\Delta_D^{\Omega} + 1)^{-1} \begin{pmatrix} I_2 & 0 \\ 0 & 0 \end{pmatrix} \right] \left( I_4 + (z+1) \left(H_\kappa^{\Omega} - \left(z + \frac{c^2}{2} \right) \right)^{-1} \right).
    \end{split}
  \end{equation*}

  This completes the proof of Theorem~\ref{theorem_nr_limit} and now we turn our attention to Corollary~\ref{cor_Faber_Krahn},
  which can be viewed as an immediate consequence of classical results on eigenvalues of Dirichlet Laplacians and 
  convergence of spectra under operator norm convergence of resolvents. For the convenience of the reader and to keep the presentation self-contained we briefly
  provide the details of the arguments.
  In the present situation, it is convenient to apply \cite[Satz~3.17~d)]{W00} or \cite[Theorem 2.3.1]{H06} about the convergence of eigenvalues of nonnegative compact operators. More precisely, let $B_c$, $c \in (c_0, \infty]$ for a suitable $c_0 \in \mathbb{R}$, be a family of compact, self-adjoint, and nonnegative operators with eigenvalues $\mu_1(B_c) \geq \mu_2(B_c) \geq \dots \geq 0$ taking multiplicities into account. If $B_c$ converges to $B_\infty$ in the operator norm, as $c \rightarrow \infty$, then by \cite[Satz~3.17~d)]{W00} for all $j \in \mathbb{N}$ also $\mu_j(B_c)$ converges to $\mu_j(B_\infty)$, as $c \rightarrow \infty$.
  We apply this result to 
  \begin{equation*}
    B_c := f\left(\left(H_\kappa^{\Omega} - \bigg(-1 + \frac{c^2}{2}\right)\right)^{-1}\bigg) \quad \text{and} \quad 
    B_\infty := f\left((-\Delta_D^{\Omega} +1)^{-1}\begin{pmatrix}I_2 & 0\\0&0 \end{pmatrix}\right),
  \end{equation*}
  where $f \in C(\R)$ is a nonnegative function such that $f(x) = 0$ for $x \in (-\infty,0]\cup [2,\infty)$, and $f(x) = x$ for $x \in (0,1]$. 
  It is clear that $B_c$ and $B_\infty$ are nonnegative bounded self-adjoint operators, and from the compactness of the resolvents of $H_\kappa^{\Omega}$ and $-\Delta_D^{\Omega}$, which holds as $\dom H_\kappa^{\Omega}$ and $\dom (-\Delta_D^{\Omega})$ are compactly embedded in $L^2(\Omega; \mathbb{C}^4)$ and $L^2(\Omega; \mathbb{C})$, respectively,
  it follows that $B_c$ and $B_\infty$ are both compact.
  
  In the following let $c>1$. For the eigenvalues $\lambda^\pm_j(H_\kappa^\Omega)$ of $H_\kappa^\Omega$ ordered as in \eqref{evs} we have
  \begin{equation*}
   \left(\lambda_j^-(H_\kappa^\Omega) +1 - \frac{c^2}{2} \right)^{-1} < 0\quad\text{and}\quad \left(\lambda_j^+(H_\kappa^\Omega) +1 - \frac{c^2}{2} \right)^{-1} \leq 1,\qquad j \in \mathbb{N};
  \end{equation*}
   cf. Corollary~\ref{corollary_H_tau_properties}.
   From the choice of $f$ it is clear that the positive eigenvalues 
   of $B_c$ are given by $\mu_j(B_c)=(\lambda_j^+(H_\kappa^\Omega) +1 - \frac{c^2}{2})^{-1}$. Similarly, $\sigma(-\Delta_D^\Omega) \subset [0, \infty)$ and the fact that two copies of 
   $(-\Delta_D^\Omega + 1)^{-1}$ appear in the definition of $B_\infty$ leads to 
   $\mu_{2j-1}(B_\infty) = \mu_{2j}(B_\infty) = (\lambda_j(-\Delta_D^\Omega) + 1)^{-1}$ for $j \in \mathbb{N}$, where
   $0 < \lambda_1(-\Delta_D^\Omega) \leq \lambda_2(-\Delta_D^\Omega) \leq \dots$ denote the discrete eigenvalues of $-\Delta_D^\Omega$ taking multiplicities into account.
   Now it follows from Theorem~\ref{theorem_nr_limit} and \cite[Theorem VIII.20]{RS72} that $B_c$ converges to $B_\infty$ in the operator norm, as $c \to \infty$. Using that all eigenvalues of $H_\kappa^\Omega$ have even multiplicity, see Corollary~\ref{corollary_H_tau_properties}~(i), we conclude with \cite[Satz~3.17~d)]{W00} from the above considerations that for any $j \in \mathbb{N}$
  \begin{equation*} 
    \mu_{2j-1}(B_c) = \mu_{2j}(B_c) = \left(\lambda_{2j-1}^+(H_\kappa^\Omega) - \frac{c^2}{2} + 1\right)^{-1}=\left(\lambda_{2j}^+(H_\kappa^\Omega) - \frac{c^2}{2} + 1\right)^{-1} 
  \end{equation*}
  tends to
  \begin{equation*} 
     \mu_{2j-1}(B_\infty) = \mu_{2j}(B_\infty) = \big(\lambda_j(-\Delta_D^\Omega) + 1\big)^{-1},\quad \text{as}\,\,c\rightarrow\infty,
  \end{equation*}
  which is equivalent to
   \begin{equation}\label{convergence_eigenvalues} 
 	\lambda_{2j -1 }^+(H_\kappa^{\Omega}) - \frac{c^2}{2} = \lambda_{2j}^+(H_\kappa^{\Omega}) - \frac{c^2}{2} \rightarrow \lambda_{j}(-\Delta_D^\Omega), \quad \text{as } c \rightarrow \infty.
 \end{equation}
  
  Eventually, to conclude  Corollary~\ref{cor_Faber_Krahn} we note that~\eqref{convergence_eigenvalues} remains true if $\Omega$ is replaced by a ball $B$ or the disjoint union of two balls $B_1 \cup B_2$. Thus, the claims follow immediately from the classical results for the Dirichlet Laplacian, which under the assumptions of Corollary~\ref{cor_Faber_Krahn} read as follows:
  \begin{itemize}
  	\item[(i)] Faber-Krahn inequality: $\lambda_{1}(-\Delta_D^B) \leq \lambda_{1}(-\Delta_D^\Omega)$  and equality holds if and only if $\Omega$ is a ball, see \cite{F23, K25} and also \cite[Theorem~3.2.1 and Remark~3.2.2]{H06}.
  	\item[(ii)] Hong-Krahn-Szegö inequality: $\lambda_{2}(-\Delta_D^{B_1 \cup B_2})\leq \lambda_{2}(-\Delta_D^\Omega)$ and equality holds if and only if $\Omega$ is the  union of two identical disjoint balls, see \cite{H54, K26} and also \cite[Theorem~4.1.1 and Remark~4.1.2]{H06}.
  	\item[(iii)] Payne-P\'{o}lya-Weinberger inequality: If $\Omega$ is connected, then 
  	\begin{equation*}
  		\frac{\lambda_{1}(-\Delta_D^{B})}{\lambda_{2}(-\Delta_D^{B})} \leq \frac{\lambda_{1}(-\Delta_D^{\Omega})}{\lambda_{2}(-\Delta_D^{\Omega})} 
  	\end{equation*}
  	and  equality holds if and only if $\Omega$ is a ball, see \cite{AB92, P55}.
  \end{itemize}

\begin{rem}\label{laterrem}
Finally, let us remark that from Theorem~\ref{theorem_nr_limit} and Corollary~\ref{cor_Faber_Krahn} one gets information about the positive part of the spectrum of $H_\kappa^{\Omega}$. Similar statements are also true for the negative part of the spectrum of $H_\kappa^{\Omega}$. Indeed, consider the self-adjoint unitary matrix $U = \big( \begin{smallmatrix} 0 & -i I_2 \\ i I_2 & 0 \end{smallmatrix} \big)$. Then, as $U \alpha_j + \alpha_j U = 0$, $j \in \{ 1, 2, 3 \}$, and  $U \beta + \beta U = 0$, it is not difficult to see that $
H_\kappa^{\Omega} = -U H_{-\kappa}^{\Omega} U$, i.e. $H_\kappa^{\Omega}$ and $-H_{-\kappa}^{\Omega}$ are unitarily equivalent. Hence, it follows from Theorem~\ref{theorem_nr_limit} that 
for $z \in \mathbb{C} \setminus (-\infty, 0]$ also
\begin{equation*}
  \begin{split}
    \left( H_\kappa^{\Omega} - \left(z - \frac{c^2}{2}\right) \right)^{-1} &= \left( -U H_{-\kappa}^{\Omega} U - \left(z - \frac{c^2}{2}\right) \right)^{-1} = -U \left(H_{-\kappa}^{\Omega} - \left(-z + \frac{c^2}{2}\right) \right)^{-1} U \\
    &\rightarrow -U (-\Delta_D^{\Omega} + z)^{-1} \begin{pmatrix} I_2 & 0 \\ 0 & 0 \end{pmatrix} U=  \big(-(-\Delta_D^{\Omega}) - z\big)^{-1} \begin{pmatrix} 0 & 0 \\ 0 & I_2 \end{pmatrix} 
  \end{split}
\end{equation*}
in the operator norm, as $c \to \infty$. This convergence is of interest by its own, but similarly as in the proof of Corollary~\ref{cor_Faber_Krahn} one can conclude 
spectral inequalities for the negative eigenvalues $\lambda_{j}^-(H_\kappa^{\Omega})$ of $H_\kappa^{\Omega}$; alternatively one can argue via the unitary equivalence $H_\kappa^{\Omega} = -U H_{-\kappa}^{\Omega} U$. More precisely, 
for a bounded $C^2$-domain $\Omega \subset \mathbb{R}^3$, a ball $B \subset \mathbb{R}^3$ with $\vert B\vert=\vert\Omega\vert$, and two identical and disjoint balls $B_1,B_2 \subset \R^3$ with $\vert B_1\vert + \vert B_2\vert=\vert\Omega\vert$ the following assertions follow for sufficiently large $c>0$:
  \begin{itemize}
  	\item[(i)] $\lambda_{j}^-(H_\kappa^{B}) \geq \lambda_{j}^-(H_\kappa^{\Omega})$ for $j \in \{1,2\}$ and equality holds if and only if $\Omega$ is a ball.
  	\item[(ii)] $\lambda_{j}^-(H_\kappa^{B_1 \cup B_2}) \geq \lambda_{j}^-(H_\kappa^{\Omega})$ for $j \in \{3,4\}$ and equality holds if and only if $\Omega$ is the union of two identical disjoint balls.
  	\item[(iii)] If, in addition, $\Omega$ is connected, then 
  	\begin{equation*}
  		\frac{\lambda_{j}^-(H_\kappa^{B})}{\lambda_{l}^-(H_\kappa^{B})} \leq \frac{\lambda_{j}^-(H_\kappa^{\Omega})}{ \lambda_{l}^-(H_\kappa^{\Omega})}, \qquad j \in \{1,2\},\, l \in \{3,4\},
  	\end{equation*}
   and  equality holds if and only if $\Omega$ is a ball.
  \end{itemize}
\end{rem}

\subsection*{Acknowledgements.}
The authors gratefully acknowledge financial support by the Austrian Science Fund (FWF): P33568-N. This publication is 
based upon work from COST Action CA 18232 MAT-DYN-NET, supported by COST (European Cooperation in Science and Technology), www.cost.eu.
The authors thank the referee for helpful comments
and remarks that led to an improvement of the manuscript.

\subsection*{Data availability statement}

Data sharing not applicable to this article as no datasets were generated or analysed during the current study.

\subsection*{Competing Interests}

The authors have no competing interests to declare that are relevant to the content of this article.


%

\bibliographystyle{abbrv}

\end{document}